\documentclass[a4paper,12pt]{article}
\usepackage{amsfonts}
\usepackage{pifont}
\usepackage{caption}
\usepackage{graphicx, subfig}
\usepackage{bm}
\usepackage{latexsym,amsmath,amssymb,cite,amsthm}
\usepackage{color,eucal,enumerate,mathrsfs}
\usepackage[normalem]{ulem}
\usepackage{amsmath}
\usepackage[pagewise]{lineno}
\usepackage[pagebackref=true, colorlinks, linkcolor=blue,  anchorcolor=blue,citecolor=blue]{hyperref}

\numberwithin{equation}{section}
\theoremstyle{plain}
\newtheorem{theorem}{Theorem}[section]
\newtheorem{corollary}[theorem]{Corollary}
\newtheorem{lemma}[theorem]{Lemma}
\newtheorem{example}[theorem]{Example}
\newtheorem{proposition}[theorem]{Proposition}
\theoremstyle{definition}
\newtheorem{definition}[theorem]{Definition}
\theoremstyle{remark}

\newtheorem{remark}[theorem]{Remark}

\newcommand\bdf{\begin{definition}}
\newcommand\bpr{\begin{proposition}}
\newcommand\brk{\begin{remark}}
\newcommand\blm{\begin{lemma}}
\newcommand\bexe{\begin{exercise}}
\newcommand\bexa{\begin{example}}
\newcommand\beqn{\begin{eqnarray*}}
\newcommand\edf{\end{definition}}
\newcommand\epr{\end{proposition}}
\newcommand\erk{\end{remark}}
\newcommand\elm{\end{lemma}}
\newcommand\eexe{\end{exercise}}
\newcommand\eexa{\end{example}}
\newcommand\eeqn{\end{eqnarray*}}

\newcommand{\ds}{\mathsf{d}}
\newcommand{\Opt}{\rm OptGeo}
\newcommand{\geo}{\rm Geo}

\newcommand{\lmt}[2]{\mathop{\lim}_{{#1} \rightarrow {#2}} }

\newcommand{\lip}[1]{{\mathrm{lip}}({#1})}

\newcommand{\lmts}[2]{\mathop{\overline{\lim}}_{{#1} \rightarrow {#2}} }
\newcommand{\lmti}[2]{\mathop{\underline{\lim}}_{{#1} \rightarrow {#2}} }

\newcommand{\Ric}{{\rm{Ricci}}}

\newcommand{\mm}{\mathfrak m}
\newcommand{\ms}{(X,\ds,\mm)}
\newcommand{\CD}{{\rm CD}}

\newcommand{\cdkn}{{\rm CD}(K, N)}

\newcommand{\rcdkn}{{\rm RCD}(K, N)}
\newcommand{\rcd}{{\rm RCD}(K, \infty)}

\newcommand{\ent}[1]{{\rm Ent}_{#1}}
\newcommand{\De}{\mathrm{D}}

\newcommand{\N}{\mathbb{N}}

\newcommand{\R}{\mathbb{R}}

\newcommand{\pr}{\mathcal{P}}

\newcommand{\Lip}{\mathop{\rm Lip}\nolimits}

\renewcommand{\d }{{\mathrm d}}

\newcommand{\dt}{{\d t}}

\newcommand{\restr}[1]{\lower3pt\hbox{$|_{#1}$}}

\newcommand{\nchi}{{\raise.3ex\hbox{$\chi$}}}

\title{\Large{\bf On the geometry of Wasserstein barycenter I}
}

\begin{document}
\author{Bang-Xian Han\thanks{School of   Mathematics,  Shandong University,  Jinan, China.  Email: hanbx@sdu.edu.cn. }
\and Deng-Yu Liu
\thanks{School of Mathematical Sciences, University of Science and Technology of China, Hefei, China. Email:  yzldy@mail.ustc.edu.cn}
\and Zhuo-Nan Zhu
\thanks{School of Mathematical Sciences, University of Science and Technology of China, Hefei, China.  Email: zhuonanzhu@mail.ustc.edu.cn}
}

\date{\today} 
\maketitle

\begin{abstract}
We study the Wasserstein barycenter problem   in the setting of   non-smooth  extended metric measure spaces.  We introduce a couple of new concepts and obtain the existence,  uniqueness, absolute continuity of the Wasserstein barycenter,  and prove  Jensen's  inequality in an abstract framework.
 This generalizes several results on Euclidean spaces, Riemannian manifolds and Alexandrov spaces,  to metric measure spaces satisfying Riemannian Curvature-Dimension condition,  as well as some extended metric measure spaces including abstract Wiener spaces and configuration spaces over Riemannian manifolds.   As a by product,  we obtain the well-posedness of the multi-marginal optimal transport problem in a very general setting.

 We also introduce a  new curvature-dimension condition, called   Barycenter-Curvature-Dimension condition. We  prove its stability in measured-Gromov--Hausdorff convergence and prove the existence of the Wasserstein barycenter  under this new condition. In addition, we get some geometric inequalities including a multi-marginal  Brunn--Minkowski inequality  and a  functional Blaschke--Santal\'o type  inequality.
\end{abstract}


\tableofcontents


\section{Introduction}
This paper has two goals.  The first  one is to extend results about Wasserstein barycenter problem from the setting of smooth Riemannian manifolds to the setting of non-smooth extended metric measure spaces.  The second one is  to give a notion for an extended  metric measure space to have Ricci curvature bounded from below, using Jensen's inequality for Wasserstein barycenters.  We refer to \cite{BBI-A}, \cite{SantambrogioBook} and \cite{V-O} for basic material on metric geometry,  optimal transport  and synthetic theory of  curvature bounds.    In the introduction, we motivate the questions that we address and we summarize the main results.

\subsection*{Wasserstein barycenter}

Let $\Omega$  be a Borel probability measure  on a complete metric space $(Y,\ds_Y)$, a barycenter of $\Omega$ is defined as a minimizer of the map $$x\mapsto\int_{Y} \ds_Y^2(x,y)\,\d \Omega(y)\in [0,  +\infty]$$ and the variance of $\Omega$ is defined as
\[
{\rm Var}(\Omega):=\inf_{x\in X} \int_{Y} \ds_Y^2(x,y)\,\d \Omega(y).
\]  In general,  we do not have  existence or uniqueness  of the minimizer.   In some cases,   for example when $(Y, \ds_Y)$ is  NPC (metric spaces of non-positive curvature),   we have existence  and uniqueness of the barycenter for sure.

A bridge, called Jensen's inequality, between convexity and probability measures was firstly established by Jensen in his seminal paper \cite{jensen1906fonctions} and had been highly concerned and extensively explored after that. {Jensen's inequality in an abstract metric space can be formulated in the following way.}
Let $\Omega$ be a probability measure with finite variance on $Y$, and $K \in \R$.  We say that a function $F: Y \to \R \cup \{+\infty\}$ is weakly (strongly)  \emph{$K$-barycentrically convex} (see Definition \ref{def:baryconv}),  if there is a (for any) barycenter $\bar{x}$ of $\Omega$, it holds the inequality
\begin{equation}\label{eq:introbc}
F(\bar{x})\leq \int_Y F(x)\,\d \Omega(x)- \frac K2 {\rm Var}({\Omega}).\tag{JI}
\end{equation}
This inequality is also known as Jensen's inequality. A fundamental problem concerning this inequality,   is to find  geodesically convex functions over a metric space such that Jensen's inequality \eqref{eq:introbc} holds. This problem has been  widely studied in the context of Riemannian manifolds \cite{emery2006barycentre, Bijan2011},  $\rm CAT$ spaces \cite{sturm2003probability, kendall1990probability, kuwae1997jensen, yokota2016convex}, Alexandrov spaces \cite{paris2020jensen} and convex metric spaces \cite{kuwae2014jensen}.

\medskip

We say that $(X,\ds)$ is an extended  metric space if $\ds : X \times X \to  [0, +\infty]$ is a symmetric function satisfying the triangle inequality, with $\ds(x, y) = 0$ if and only if $x = y$.
Given a Hausdorff topology $\tau$ on $X$ and the Borel $\sigma$-algebra $\mathcal B(\tau)$. {Denote by $\mathcal P(X)$  the set of Radon probability measures on $X$, and by $\mathcal P_2(X,\ds)$  the set of  \emph{probability measures  with finite variance} (or second order moment)}, i.e. $\mu\in  \mathcal P_2(X,\ds)$ if and only if $\mu\in \mathcal P(X)$ and $\int \ds^2(x, x_0)\,\d \mu(x)<\infty$ for some 
 $x_0\in X$. 

Consider the Wasserstein space $\mathcal W_2:=(\pr(X),W_2)$ equipped with the   so-called  $L^{2}$-Kantorovich-Wasserstein or $L^{2}$-optimal transport distance:
 \[
 W_2^2(\mu, \nu):=\inf_\Pi \int \ds^2(x, y)\,\d \Pi(x, y),
 \]
  where the infimum is taken among all transport plans $\Pi\in\pr(X\times X)$ with marginals $\mu$ and $ \nu$.
It is  known that $W_2$ is  an extended metric on $\pr(X)$ (see \cite[Proposition 5.3]{ambrosio2016optimal}), and $W_2$ is  a metric on $\pr_2(X, \ds)$ if $(X, \d)$ is a metric space (see \cite[Theorem 2.2]{AG-U}). Therefore it makes sense to talk about barycenters in the Wasserstein space.
This problem, called Wasserstein barycenter problem,   draws particular interests in the optimal transport theory and its related fields, as it gives a natural but non-linear way to interpolate between a distribution of measures.

\begin{definition}[Wasserstein barycenter]
	Let $(X,\ds)$ be an extended metric space and let $\Omega\in \pr_2(\pr(X), W_2)$ be a Radon probability measure on the Wasserstein space  $(\pr(X), W_2)$ with finite variance. We call $\bar{\nu}\in \pr(X)$ a Wasserstein barycenter of $\Omega$ if
	$$\int_{\pr(X)} W_2^2(\mu,\bar{\nu})\,\d\Omega(\mu)=\min_{\nu\in \pr(X)}\int_{\pr(X)} W_2^2(\mu, \nu)\,\d \Omega(\mu).$$
\end{definition}

It is worth to mention that, if   $(X,  \ds)$ is  a geodesic space,  any Wasserstein barycenter of  $\Omega=\frac 12 \delta_{\mu_0}+\frac 12\delta_{\mu_1}$ is exactly a mid-point of $\mu_0, \mu_1$ in the Wasserstein space.

\medskip

The study of the Wasserstein barycenter problem has got a lot of attention from experts in various fields in the last decade.   In mathematical  economic, Carlier--Ekeland \cite{carlier2010matching}  studied team matching problems by considering the interpolation between multiple probability measures, which  can be seen as a generalization of Wasserstein barycenter with square of Wasserstein distance replaced by the general cost. From the perspective of statistic, Wasserstein barycenter can be understood as the mean, which is a central topic when dealing with a large number of data, of a data sample composed of probability measures and thus shows its application value in  data science \cite{altschuler2022wasserstein},  image processing \cite{rabin2012wasserstein} and statistics \cite{bigot2013consistent}.  

In  metric measure geometry,   Agueh--Carlier \cite{AguehCarlier} in the Euclidean setting,   Kim--Pass \cite{KimPassAIM}  and Ma \cite{ma2023} in the Riemannian setting,  Jiang \cite{JiangWasserstein} in the Alexandrov spaces,   established the existence and uniqueness of  Wasserstein barycenter, and the absolutely continuity of the Wasserstein barycenter with respect to the canonical reference measures in these cases.  \emph{As a consequence}, they prove the following Wasserstein Jensen's inequality, which characterizes the convexity of certain functionals via Wasserstein barycenter  (cf. Definition \ref{def:baryconv}): 

Let  $\mathcal F: \pr_2(X, \ds) \to \R \cup \{+\infty\}$  be a $K$-displacement convex functional in the sense of McCann \cite{mccann1997convexity},     $\Omega$ be a probability measure on $\mathcal{P}_2(X, \ds)$  and  $\bar{\mu}$ be its unique barycenter.   It holds 
 \begin{equation}\label{jensen}
\mathcal F(\bar{\mu})\leq \int_{\mathcal{P}_2(X, \ds)}\mathcal F(\mu)\,\d\Omega(\mu)-\frac{K}{2}\int_{\mathcal{P}_2(X, \ds)}W_2^2(\bar{\mu},\mu)\,\d\Omega(\mu).\tag{WJI}
\end{equation}
 
\medskip

The first  goal of this paper is to extend results in \cite{AguehCarlier, KimPassAIM, ma2023, JiangWasserstein},   including the key properties \emph{existence, uniqueness, and regularity of Wasserstein barycenter},  to the setting of  extended metric measure spaces.  
The main difficulties to achieve this aim,  comparing with  the previous results,   lies in the fact that 
\begin{itemize}
\item  all the known existence  results depend on local compactness,  which is not available for non-compact spaces such as infinite-dimensional RCD spaces or abstract Wiener spaces;
\item  in the Euclidean setting,   regularity results relies heavily on    the special geometric and algebraic  structure of the space;
\item in the case of Riemannian manifolds and Alexandrov spaces,    sectional curvature bounds play important roles.
\end{itemize}   
Therefore the known methods  are difficult to be realized, in the setting of general extended metric measure space. To overcome these difficulties,   we will show that
 \begin{quote}
 \emph{regularity of the Wasserstein barycenter is a consequence,  other than a necessary condition of the Jensen's inequality}.
 \end{quote}
Furthermore,   as   one of our main innovations in this paper, we will show 
 \begin{quote}
 \emph{Jensen's inequality plays the role of  ``a priori estimate" in the theory of partial differential equations}.
 \end{quote} 
 In other words,   well-posedness of the Wasserstein barycenter problem can be deduced from Jensen's inequality.
 
Our idea is to  take advantage of gradient flow theory to study Jensen's inequality \eqref{eq:introbc}, and   Wasserstein barycenter problem.   Motivated by a work of Daneri and Savar{\'e} \cite{daneri2008eulerian}, we realize that a special formulation of gradient flows, called \emph{Evolution Variational Inequality}, implies  Jensen's inequality with respect to the barycenter. Consequently, we give a direct, simple proof of Wasserstein Jensen's inequality  by using the theory of gradient flows on  Wasserstein spaces and give a sufficient condition for the validity of Jensen's inequality. In particular, our strategy is synthetic,  dimension-free, independent of the properness of $(X, \d)$ and improve the result in \cite[Theorem 7.11]{KimPassAIM}.

More precisely, let $F$ be a lower semi-continuous function on a general extended  metric space $(X, \ds)$,  such that any point $x_0$  with finite distance to the domain of $F$ is the starting point of an ${\rm EVI}_K$-type gradient flow  $(x_t)$ of  $F$ satisfying 
\[
\lmti{t}{ 0} F(x_t) \geq F(x_0),~~ \lim_{t\to 0} \ds(x_t, y) \to \ds(x_0, y),~~ ~~\forall y\in  \overline{  {\rm D}(F)}
 \]
 and
\[
\frac {\d^+}{\dt} \frac 12 \ds^2(x_t, y)+\frac K2 \ds^2(x_t, y)\leq F(y)-F(x_t),~~~\forall t>0
\]
for all  $y \in  {\rm D}(F)$ satisfying $\ds(y, x_t) <\infty$ for some (and then all) $t \in (0, \infty) $, where $\frac {\d^+}{\dt} $ denotes the upper right derivative.
Then for a barycenter $\bar{x}$ of $\Omega$,  by considering the gradient flow from $\bar x$,  we  can establish  Jensen's type inequality \eqref{eq:introbc} for this function $F$.

\bexa The existence of  the  ${\rm EVI}_K$-type gradient flow of a $K$-convex function,   is    valid on Euclidean spaces, Hilbert spaces, Alexandrov spaces, ${\rm CAT}$ spaces.  Concerning the Wasserstien space $(\mathcal{P}(X),W_2)$, this property is satisfied  when the underling space $X$ is an Euclidean space, a smooth Riemannian manifold with uniform Ricci lower bound, an Alexandrov space with  curvature bounded from below,  a Wiener space (see \S \ref{5.1.2} or an $\rm RCD$ space (see \S \ref{SS:CDDef}). We refer the readers to \cite{ambrosio2005gradient, muratori2020gradient, ambrosio2016optimal} for more discussions on this topic.  
\eexa

Concerning Wasserstein barycenter problem, we prove the following  theorems without local compactness or any other fine local structure of the underling space $(X,\ds)$. Note that if $(X, \ds)$ is an extended metric space, the definition of  $\pr_2(X, \ds)$ may depends on the choice of base point, which is less meaningful. Thus we will deal with metric spaces and extended metric spaces separately.

First of all,   \emph{without assuming the existence of  Wasserstein barycenter},  we prove Wasserstein Jensen's inequality.  As a corollary, we get the absolute continuity of  Wasserstein barycenter(s).

\begin{theorem}[Wasserstein Jensen's inequality, Theorem \ref{WJI}]\label{th1:intro}
Let $K\in \mathbb{R}$, $(X,\ds,\mm)$ be an extended metric measure space,  $\Omega$ be a  probability measure over $\mathcal{P}(X)$ with finite variance. Assume one of the following conditions is satisfied:
\begin{itemize}
\item [{\bf A.}]  $(X, \ds, \mm)$ is an $\rcd$ metric measure space and $\Omega$ is concentrated on $\mathcal{P}_2(X,\ds)$;
\item [{\bf B.}]  $\mm$ is a probability measure,  any  ${\mu} \in \pr(X)$ which has finite distance from $\De( {\rm Ent}_{\mm})$ is the starting point of an ${\rm EVI}_K$ gradient flow of the relative entropy ${\rm Ent}_{\mm}$ in the Wasserstein space.
\end{itemize}

	Then for any barycenter $\bar{\mu}$ of $\Omega$,  the  Wasserstein Jensen's inequality follows:
	\begin{equation}\label{intro:WJI1}
	{\rm Ent}_{\mm}(\bar{\mu})\leq \int_{\mathcal{P}(X)}{\rm Ent}_{\mm}(\mu)\,\d\Omega(\mu)-\frac{K}{2}\int_{\mathcal{P}(X)}W_2^2(\bar{\mu},\mu)\,\d\Omega(\mu).
	\end{equation}
 Furthermore, if 
	$$
	\int_{\mathcal{P}(X)}{\rm Ent}_{\mm}(\mu)\,\d\Omega(\mu)<\infty,
	$$
	 then the entropy of the barycenter of $\Omega$ is finite. In particular, the barycenter  is absolutely continuous with respect to $\mm$.
\end{theorem}

{As a corollary},   we get the following  existence and uniqueness  theorem.

\begin{theorem}[Existence and uniqueness of barycenter, Theorem \ref{th:existenceRCD} and Theorem \ref{uniqueness in infinite space}]
	Assume  $(X,\ds,\mm)$   fulfil the hypothesis in Theorem \ref{th1:intro}.   Let $\Omega$ be a  probability measure over $\mathcal{P}(X)$ with finite variance and $$
	\int_{\mathcal{P}(X)}{\rm Ent}_{\mm}(\mu)\,\d\Omega(\mu)<\infty.
	$$  Then  $\Omega$ has a barycenter.
	
	Furthermore,  if $\ms$ satisfies the weak Monge property,    then $\Omega$ has a unique barycenter.

\end{theorem}

\medskip

Similarly,  using a 
finite dimensional variant of the ${\rm EVI}_{K}$ gradient flow,  called ${\rm EVI}_{K,N}$ gradient flow,  we prove a finite dimensional variant of 
Jensen's inequality, which seems new in the literature even on $\R^n$.

\begin{theorem}[Jensen's inequality with dimension parameter, Theorem \ref{JIEVI -K,N}, Corollary \ref{Jensen's equality finite} and Corollary \ref{entropy finite EVI -k,N}]\label{th2:intro}
Let $N\in [1,\infty)$. Let $(X,\ds,\mm)$ be an ${\rm RCD}(K,N)$  metric measure space and $\Omega$ be a Borel probability measure over $\mathcal{P}_2(X, \ds)$ with finite variance. Then the  dimensional Wasserstein Jensen's inequality follows:
\begin{equation}\label{intro: JIKN}
\int \frac{W_2(\bar{\mu},\mu)}{s_{K/N}(W_2(\bar{\mu},\mu))}U_N(\mu)\,\d\Omega(\mu)\leq  U_N(\bar{\mu}) \int \frac{W_2(\bar{\mu},\mu)}{t_{K/N}(W_2(\bar{\mu},\mu))}\,\d\Omega(\mu)
\end{equation}
for any  barycenter $\bar{\mu}$  of $\Omega$,   where  $s_{K/N}$ and  $t_{K/N}$ are distortion coefficients,  $U_N(\mu)=e^{-\frac{{\rm Ent}_{\mm}(\mu)}{N}}$.

In particular, 
if   the set $\left\{\mu:  {\rm Ent}_{\mm}(\mu)<\infty\right\}$ has positive $\Omega$-measure. Then  the barycenter of $\Omega$ is unique,  and this barycenter is absolutely continuous with respect to $\mm$.
\end{theorem}
\subsection*{Barycenter-Curvature-Dimension condition}
The second goal of this paper is to provide a new notion, in terms  of Wasserstein barycenters, for an extended metric measure spaces to have synthetic  Ricci curvature bounded below and dimension bounded above.

There are various approaches to  extend notions of curvature  from smooth Riemannian manifolds to more general spaces.   A  good notion of a length space having ``sectional curvature bounded below" is Alexandrov space,  which is defined in terms of  Toponogov's comparison theorem concerning geodesic triangles (cf.  \cite{BBI-A} and \cite{AKPbook}).
More recently, a notion of  a metric measure space $\ms$ having ``Ricci curvature bounded below"  was introduced by Sturm\cite{S-O1, S-O2} and Lott--Villani \cite{Lott-Villani09} independently. This new synthetic  curvature-dimension condition is defined  in terms of a notion of geodesic convexity for functionals on the Wasserstein space $(\pr_2(X, \ds), W_2)$,  called displacement convexity, introduced by  McCann in his celebrated paper \cite{mccann1997convexity}.   This theory, called  Lott--Sturm--Villani theory today,    has been widely used in the study of  Ricci limit spaces, geometric and functional inequalities and many areas in applied mathematics.  Recently, Lott--Sturm--Villani's   curvature-dimension condition has been generalized to the setting of extended metric measure spaces, see \cite{ambrosio2016optimal, EH-Configure, DSHS,SS-configueII} and the references therein.

{From \cite{KimPassAIM},  we can see that the Wasserstein Jensen's inequality of certain functionals characterizes  the lower Ricci curvature bounds  of  Riemannian manifolds, in the same spirit of Cordero-Erausquin, McCann and Schmuckenschl{\"a}ger \cite{CMS01}}.  In Section \ref{main} we find more examples satisfying the Wasserstein Jensen's inequality, including non-compact spaces such as RCD spaces and extended metric spaces such as abstract Wiener spaces.
Based on these results, it is natural to provide a notion for general extended metric measure spaces to have   Ricci curvature bounded from below,  via the Wasserstein Jensen's inequality. This approach is surely compatible with Lott--Sturm--Villani theory and  has its own  highlights and interests.
 

\medskip

\begin{definition}[${\rm BCD}(K,\infty)$  condition, Definition \ref{def:bcd}]
	
	Let $K \in \R$.  We say that an \emph{extended  metric measure  space}  $(X,\ds,\mm)$ verifies ${\rm BCD}(K,\infty)$ condition,   if for any  probability measure $\Omega\in \pr_2(\pr(X), W_2)$,  concentrated on \emph{finitely many} measures, there exists a barycenter $\bar{\mu}$ of  $\Omega$ such that the following Jensen's  inequality holds:
	\begin{equation}
	{\rm Ent}_{\mm}(\bar{\mu})\leq\int_{\mathcal{P}(X)}{\rm Ent}_{\mm}(\mu)\,\d\Omega(\mu)-\frac{K}{2}{\rm Var}(\Omega).
	\end{equation}

\end{definition}

Not surprisingly,   the family of compact  metric measure spaces satisfying barycenter curvature-dimension condition is closed in measured Gromov--Hausdorff topology. 

\begin{theorem}[Stability in measured Gromov--Hausdorff topology, Theorem \ref{measured gromov hausdorff limit}]
	Let $\left\{(X_i, \ds_i, \nu_i)\right\}_{i=1}^{\infty}$ be a sequence of compact ${\rm BCD}(K,\infty)$ metric measure spaces with $K\in\mathbb{R}$. If  $(X_i, \ds_i, \nu_i)$ converges to $(X, \ds, \nu)$ in the measured Gromov--Hausdorff sense as $i\rightarrow \infty$, then $(X, \ds, \nu)$ is also a ${\rm BCD}(K,\infty)$ space. 
\end{theorem}

Our theory has various applications, one of the most important and surprising ones is the resolvability  of the Wasserstein barycenter problem on BCD spaces.
Note that a BCD space does not necessarily have a Riemannian structure, the scope of our theorem is  far beyond Ricci-limit spaces and RCD spaces.

\begin{theorem}[Existence of Wasserstein barycenter, Theorem \ref{Existence of Wasserstein barycenter in barycenter space}]
	Let $(X,\d,\mm)$ be an extended metric measure space satisfying ${\rm BCD}(K,\infty)$ curvature-dimension condition, $\Omega$ be a probability measure on $\pr(X)$ satisfying $${\rm Var}(\Omega)<\infty~~~~\text{and}~~~\int_{\pr(X)}{\rm Ent}_\mm(\mu)\,\d\Omega(\mu)<\infty.$$ Then  $\Omega$ has a barycenter if one of the following conditions holds:
\begin{itemize}
\item [{\bf A.}]  $(X, \ds, \mm)$ satisfies the exponential volume growth condition and $\Omega$ is concentrated on $\mathcal{P}_2(X, \ds)$;
\item [{\bf B.}]  $\mm$ is a probability measure.
\end{itemize} 
\end{theorem}
\medskip

In the definition of ${\rm BCD}(K,\infty)$  condition, there is no restriction on the dimension,  so it is  a dimension-free condition.  This indicates that we need to specify a dimension parameter  in order to define a ``finite dimensional" BCD space. Following Theorem \ref{th2:intro}, it is natural to define ${\rm BCD}(K,N)$  condition for metric measure spaces as follows.  

\begin{definition}[${\rm BCD}(K,N)$  condition, Definition \ref{def:bcdkn}]
	
	Let $K \in \R, N>0$.  We say that a   metric measure  space  $(X,\ds,\mm)$ verifies ${\rm BCD}(K,N)$ condition,   if for any  probability measure $\Omega\in \pr_2(\pr(X), W_2)$ with  {finite} support, there exists a barycenter $\bar{\mu}$ of  $\Omega$ such that the following Jensen-type  inequality holds:
\begin{equation}
\int \frac{W_2(\bar{\mu},\mu)}{s_{K/N}(W_2(\bar{\mu},\mu))}U_N(\mu)\,\d\Omega(\mu)\leq  U_N(\bar{\mu}) \int \frac{W_2(\bar{\mu},\mu)}{t_{K/N}(W_2(\bar{\mu},\mu))}\,\d\Omega(\mu),
\end{equation}
where $U_N(\mu)=e^{-\frac{{\rm Ent}_{\mm}(\mu)}{N}}$.

\end{definition}
 
We will see in \S  \,\ref{sect:app} that BCD condition implies several geometric and functional inequalities. In addition,  by letting one (or more) marginal measure be Dirac mass,  we can propose  a variant of ``Measure Contraction Property"  \`a la Ohta--Sturm (cf. \cite{Ohta-MCP, S-O2}),  in the setting of  ${\rm BCD}(K,N)$.  We will study  the geometric  and analysis consequence of this property in a forthcoming paper \cite{HLZ-BC2}.
\subsection*{Multi-marginal optimal transport problem}
Given $\mu_1,...,\mu_n \in \pr (X)$ and a lower semi-continuous cost function $c: X^n \to \R$. The multi-marginal optimal transport problem of Monge type is to minimize 
	\begin{equation}\label{MP}
		\inf_{ T_2,\dots,T_n} \int_{X} c(x_1,T_2(x_1),\dots,T_n(x_1)) \, \d\mu_1(x_1),\tag{MP}
	\end{equation}
among  $(n-1)$-tuples of map $(T_2,\dots, T_n)$, such that for each $i=2,\dots,n$, the map $T_i:X\to X$ pushes the measure $\mu_1$ forward to $\mu_i$;  that means, for any Borel $A\subseteq X$, $\mu_1(T_{i}^{-1}(A))=\mu_i(A)$.
	
The multi-marginal optimal transport problem of Kantorovich type is to solve 
\begin{equation}\label{KP}
	\inf_{ \pi\in \Pi(\mu_1,\dots,\mu_n)} \int_{X^n} c(x_1,\dots,x_n) \, \d\pi (x_1,\dots,x_n),\tag{KP}
\end{equation}
where the infimum is taken over all probability measures $\pi$ on $X^n$ whose marginals are $\mu_1,\dots,\mu_n$.

When $n = 2$ and $c(x_1, x_2)=\ds(x_1, x_2)$ or $\ds^2(x_1, x_2)$, \eqref{MP} and \eqref{KP} correspond to the Monge and Kantorovich problems in classical optimal transport theory respectively. This problem has been studied extensively over the past three decades,
which   is one of the fundamental problems in the study of  spaces satisfying synthetic curvature-dimension condition \`a la Lott--Sturm--Villani, see  \cite{Gigli12a, RS-N, cavalletti2017optimal,GRS-O}. In particular,during the study of this problem, many important by-products have been produced, including the notion of essentially non-branching and qualitatively non-degenerate (cf.  \cite{Kell-Transport}.

 We will prove uniqueness and  resolvability of  Monge problem \eqref{MP}  for the multi-marginal optimal transport problem in the setting of metric measure space, with cost function
 \begin{equation}\label{intro:mm}
c(x_1,\dots, x_n):= \inf_{y\in X}\sum_{i=1}^n \frac{1}{2} \ds^2(x_i,y).
 \end{equation}
  Alternatively, we will show that $\eqref{MP}=\eqref{KP}$.
  It is worth mentioning that this problem was solved by Gangbo and Swiech \cite{gangbo1998optimal} in the Euclidean setting, and generalized by Kim and Pass \cite{kim2015multi} in the Riemannian setting and Jiang \cite{JiangWasserstein} in the Alexandrov spaces.  The  link between multi-marginal optimal transport with cost function \eqref{intro:mm} and the Wasserstein barycenter  problem associated with the measures $\mu_1,\mu_2,\dots, \mu_n$ was discovered  by Agueh and Carlier \cite{AguehCarlier} in the Euclidean case.

  \medskip

 Our first theorem is the following. 
 
 \begin{theorem}[Existence and uniqueness of multi-marginal optimal transport map, Theorem \ref{MOT}]
	Let $(X,\ds, \mm)$ be a metric measure space. Then the multi-marginal optimal transport problem of Monge type, associated with the cost function $c(x_1,\dots,x_n)$,   has a unique solution,  if one of the following conditions holds:
	\begin{itemize}
	\item [{\bf A.}]  $(X,\ds,\mm)$ is an ${\rm RCD}(K,N)$ space, and $\mu_1\ll \mm$.
	
	\item [{\bf B.}]  $(X,\ds,\mm)$ is an ${\rm RCD}(K,\infty)$ space, and $\mu_i\ll \mm,i=1,\dots,n$.
	\end{itemize}
\end{theorem}
 
 This theorem is, to the best of our knowledge, the first of this kind for multi-marginal optimal transport problems without using any of the local structure of the underling space; and the first  result  concerning infinite-dimensional spaces. Furthermore, we  prove the existence, uniqueness and absolute continuity of the Wasserstein barycenter of measure with finite support,  \emph{without the  finite entropy  condition}. This generalized some results concerning Wasserstein geodesics,  by T. Rajala and his co-authors \cite{R-I, R-IG,  RS-N, gigli2016optimal}, to  Wasserstein barycenters.
 
\begin{theorem}[Theorem \ref{absolutely continuous}]
	Let $(X,\ds,\mm)$ be a metric measure space. Assume $\mu_1,\dots,\mu_n\in \pr_2(X, \ds)$, then there exists a unique Wasserstein barycenter $\bar{\mu}$ and it is absolutely continuous with respect to $\mm$ if one of the following conditions holds:
	\begin{itemize}
	\item [{\bf A.}]  $(X,\ds,\mm)$ is an ${\rm RCD}(K,N)$ space, and $\mu_1\ll \mm$.	
	\item [{\bf B.}]   $(X,\ds,\mm)$ is an ${\rm RCD}(K,\infty)$ space, and $\mu_i\ll \mm,i=1,\dots,n$.
	\end{itemize}
\end{theorem}
 
\bigskip

\noindent {\bf Organization of the paper:}  In Section \ref{sect2} and \ref{pre}, we collect some preliminaries in the theory of metric measure spaces, optimal transport and curvature-dimension condition. In Section \ref{section:existence} we introduce some general results,  concerning  existence and  uniqueness of the Wasserstein barycenter.   Most of the results in this part are independent of the  curvature-dimension conditions,   and are therefore of independent interest.  Section \ref{main} is devoted to proving the main theorems.  In Section \ref{bcd}, we introduce the concept of Barycenter Curvature-Dimension condition, and apply our  theory to prove some geometric  inequalities.
More detailed descriptions appear at the beginning of each section.

\bigskip

\noindent \textbf{Declaration.}
{The  authors declare that there is no conflict of interest and the manuscript has no associated data.}

\medskip

\noindent \textbf{Acknowledgement}:   This  work is supported by  the Young Scientist Programs  of the Ministry of Science \& Technology of China ( 2021YFA1000900, 2021YFA1002200),  and NSFC grant (No.12201596). The authors thanks  Emanuel Milman and Lorenzo Dello Schiavo for their interest and suggestions on this paper.

\section{Geodesic and barycenter}\label{sect2}
Let us summarize some definitions and basic results about (extended) metric spaces. For proofs and further details we refer to the text books \cite{AT-T} and \cite{BBI-A}. Throughout this section,  we denote by $X=(X, \tau)$ a Hausdorff topological space.
\subsection{Geodesic spaces}

Let $\ds$ be a  metric on $X$, this means $\ds : X \times X \to  [0, +\infty)$ is a symmetric function satisfying the triangle inequality, with $\ds(x, y) = 0$ if and only if $x = y$. Let $\gamma$ be a curve in $X$, i.e.  a continuous map from $[0,1]$ to $X$,  its length is defined by
\[
L(\gamma):=\sup_{J\in \N} \sup_{0=t_0\leq t_1\leq ...\leq t_J=1} \sum_{j=1}^J \ds(\gamma_{t_{j-1}}, \gamma_{t_j}).
\]
Clearly $L(\gamma) \geq \ds(\gamma_0, \gamma_1)$.
We say that $X$ is a length space,  if the distance between two points $x,y \in X$ is the infimum of the lengths of curves from $x$ to $y$. Such a space is path connected.

We denote by 
$$
\geo(X) : = \Big\{ \gamma \in C([0,1], X):  \ds(\gamma_{s},\gamma_{t}) = |s-t| \ds(\gamma_{0},\gamma_{1}), \text{ for every } s,t \in [0,1] \Big\}
$$
the space of constant speed geodesics,  equipped with the canonical supremum norm. The metric space $(X,\d)$ is called a {geodesic space} if for each $x,y \in X$ 
there exists $\gamma \in \geo(X)$ so that $\gamma_{0} =x, \gamma_{1} = y$. This curve is \emph{not} required to be unique.

The following lemma is a characterization of geodesic space using midpoints.
\begin{lemma}
A complete metric space $(X,\ds)$ is a geodesic space if and only if for any $x, y\in X$,  there is $z \in X$ such that
\[
\ds(x,z) = \ds(z,y) = \frac 12 \ds(x,y).
\]
Any point $z \in X$ with the above properties will be called midpoint of $x$ and $y$.
\end{lemma}
\medskip
In addition, we have the following characterization of the completeness of a length space.
\begin{lemma}[Hopf--Rinow]
Let $(X,\ds)$ be a complete length space, then:
\begin{enumerate}
\item  [(1)] The closure of $B_r(x)$,  the open ball of radius r around $x \in X$, is the closed ball $\{y\in X: \ds(x, y)\leq r\}$;
\item [(2)] $X$ is locally compact if and only if each closed ball  is compact;
\item [(3)] If $X$ is locally compact, then it is a geodesic space.
\end{enumerate}
\end{lemma}

\subsection{Extended metric spaces}
In this subsection we introduce the  notion of extended metric space.  Abstract Wiener spaces and configuration spaces over Riemannian manifolds are particular examples of extended metric spaces.

It can be seen from \cite{ambrosio2016optimal} that most of the analytic tools in metric spaces
can be extend in a natural way to extended metric spaces. 

\bdf [Extended metric spaces] We say that $(X,\ds)$ is an extended metric space if $\ds : X \times X \to  [0, +\infty]$ is a symmetric function satisfying the triangle inequality, with $\ds(x, y) = 0$ if and only if $x = y$.
\edf

Since an extended metric space can be seen as the disjoint union of the equivalence classes induced by the equivalence relation
\[x \backsim y \iff \ds(x, y) < +\infty,
\]
and since any equivalence class is indeed  a metric space, many results and definitions extend with no effort to extended metric spaces. For example, we say that an extended metric space $(X,\ds)$ is complete (resp. geodesic, length,...) if all metric spaces $X[x] = \{y : y \backsim x\}$ are complete (resp. geodesic, length,...). 

\subsection{Barycenter spaces}
Let $(X, \ds)$ be an extended metric space and $\tau$ be a Hausdorff topology on $X$. We assume that $\tau$ and $\ds$ are compatible, in the sense of \cite[Definition 4.1]{ambrosio2016optimal}. In this case, we say that $(X, \tau, \ds)$ is an extended metric-topological space.

 Let $\mathcal B(\tau)$ be the Borel $\sigma$-algebra of $\tau$ and let $\pr(X)$ be the set of  Radon probability measures on $X$. Denote by $\pr_2(X, \ds)$ the set of $\mu \in  \pr(X)$  such that $$\int \ds^2(x_0, y)\,\d \mu(y) < \infty~~\text{for some}~~x_0 \in  X.$$
In general,  the choice of such $x_0$ is not arbitrary. If $(X, \d)$ is a metric space,  $\int \ds^2(x_0, y)\,\d \mu(y) < \infty$ for some $x_0\in X$  if and only if $\int \ds^2(x, y)\,\d \mu(y) < \infty$ for any $x\in X$. For $\mu\in \pr(X)$, the value  ${\rm Var}(\mu):=\inf_{x\in X}\int \ds^2(x, y)\,\d \mu(y)\in [0,+\infty]$ is called the \emph{variance} of $\mu$. We can see that $\mu\in \pr_2(X, \ds)$ if and only if ${\rm Var}(\mu)<\infty$ (i.e.  $\mu$  has finite variance).

We also denote by $\pr_0(X)$ the set of all $\mu \in  \pr (X)$ of the form $\mu= \frac 1n \sum_{i=1}^n \delta_{x_i}$
with finite points $x_i \in X$. Here and henceforth, $\delta_x$ denotes the Dirac
measure on the point $x \in  X$. 

\bdf[Barycenter]
 Let $(X, \ds)$ be an extended metric space and let $\mu\in \pr_2(X, \ds)$ be a probability measure with finite variance. We call $\bar{x}\in X$ a barycenter of $\mu$ if 
 \[
 \int_{X} \ds^2(\bar{x},z)\,\d\mu(z)=\min_{x\in X}\int \ds^2(x, y)\,\d \mu(y)<\infty.
 \]
\edf

\medskip

Now we can introduce the notion of barycenter space. We remark that there is a different notion of ``barycenter metric space", as a generalization of Hadamard space, introduced by Lee and Naor in \cite{LeeNaor2005}.

\begin{definition}[Barycenter space]\label{def:bcspace}
We say that an extended metric space is a \emph{barycenter  extended  metric space},  or \emph{barycenter space} for simplicity, if   any $\mu \in \pr_0(X)$ with finite variance has a barycenter.
\end{definition}

\medskip
Here we list a couple of examples which are barycenter spaces. More discussions  can be found in \cite{sturm2003probability}, \cite{NaorSilberman2011}, \cite{kuwae2014jensen} and  \cite{Ohtabc, Ohta2007convex}.
\begin{example}\label{example1}
The following spaces are barycenter spaces:
\begin{itemize}
\item proper spaces,  or locally compact geodesic spaces;
\item metric spaces with non-positive curvature (or {\rm NPC} spaces), this includes complete, simply connected Riemannian manifold with non-positive (sectional) curvature and Hilbert spaces;

\item uniform convex metric spaces, introduced by James A. Clarkson \cite{Clarkson1936} in 1936, this includes uniformly convex Banach spaces such as $L^p$ spaces with $p>1$;
\item abstract Wiener space $(X,  H, \mu)$ where 
$X$ is a Banach space that contains  a Hilbert space 
$ H$ as a dense subspace, equipped with the Cameron--Martin distance 
\begin{equation*}
\ds_ H(x, y):=
\left \{\begin{array}{ll}
\|x-y\|_ H &\text{if}~ x-y\in  H\\
+\infty &\text{otherwise}.
\end{array}\right.
\end{equation*} 
\item Wasserstein spaces over  Riemannian manifolds \cite{KimPassAIM} and  Alexandrov spaces \cite{JiangWasserstein}.
\end{itemize}
\end{example}

\subsection{Convex functions}
By definition, a subset $\Omega \subset  X$ is convex if for any $x, y\in \Omega$,  there {\bf is} a  geodesic from $x$ to $y$ that lies entirely in $\Omega$. It is totally convex
if for any $x, y\in \Omega$,  {\bf any} geodesic in $X$ from $x$ to $y$ lies in $\Omega$. Given $K \in \R$, a function $F: X \to \R \cup \{+\infty\}$ is said to be weakly (strongly) \emph{$K$-geodesically convex} if there is  (for any) geodesic $(\gamma_t)_{t\in [0,1]}$ and any $t\in [0,1]$,
\[
F(\gamma_t)\leq t F(\gamma_1)+(1-t)F(\gamma_0)- Kt(1-t)\ds^2(\gamma_0, \gamma_1).
\]
In the case when $(X, \ds)$ is the Euclidean space $\R^n$ equipped with the Euclidean norm, and $F \in  C^2(X)$, this is equivalent to say that ${\rm Hess}_F \geq  K$.

With the notion of barycenter, we introduce a notion called barycenter convexity. It can be seen that this convexity is equivalent to the geodesic convexity, if  $(X, \d)$ is the Euclidean space.

\bdf[Barycenter convexity]\label{def:baryconv}
Let $(X, \ds)$ be a barycenter space.
Given $K \in \R$, a function $F: X \to \R \cup \{+\infty\}$ is said to be  weakly (strongly)  \emph{$K$-barycentrically convex} if for any  $\mu\in \pr(X)$  with finite variance,  there is a (for any) barycenter $\bar{x}$ of $\mu$,  such that
\[
F(\bar{x})\leq \int F(x)\,\d \mu(x)- \frac K2 {\rm Var}({\mu}).
\] 
\edf

\section{Geometry of the Wasserstein space}\label{pre}
In this section,  we add a  reference measure  in an extended metric-topological space as follows.
\bdf[Extended metric measure space, cf. \cite{ambrosio2016optimal}] We say that $\ms$ is an extended metric measure space if:
\begin{itemize}
\item [(a)] $(X, \tau,  \ds)$ is an extended metric-topological space and $(X, \ds)$ is complete;
\item [(b)] $\mm$ is a  non-negative Radon probability measure on $(X, \mathcal B(\tau))$ with full support.
\end{itemize}

\edf

\subsection{Wasserstein barycenter and multi-marginal optimal transport}
Given an extended metric-topological space  $(X, \tau,\ds)$.
The $L^{2}$-Kantorovich-Wasserstein,   or called $L^{2}$-optimal transport distance,   $W_2(\mu,\nu)$ between $\mu, \nu\in \pr (X)$ is given by
\begin{equation}\label{eq:W2def}
	W_2^2(\mu,\nu) := \inf_{ \pi\in\Pi(\mu, \nu)} \int_{X\times X} \ds^2(x,y) \, \d\pi(x, y),
\end{equation}
where the infimum is taken over all probability Radon measures $\pi$ on $X\times X$ whose marginals are 
$\mu$ and $\nu$. It can be seen that $W_2$ is also an extended metric on $\pr(X)$ (in fact, $(\mathcal{P}_2(X, \ds), W_2)$ is a metric space if $(X, \ds)$ is a metric space,  $(\mathcal{P}_2(X, \ds), W_2)$ is  geodesic if $(X,\ds)$ is geodesic).   Therefore,   it makes sense to talk about barycenters in the Wasserstein space $\mathcal W_2:=(\pr(X), W_2)$.  Note that  we also need to specify  a Hausdorff topology $\bar \tau$  and a Borel set $\mathcal B(\bar \tau)$ on  $\pr(X)$ in advance,    so that the following definition is well-posed.   In this paper we always assume that,  for any $\mu \in \pr(X)$,  the  single-point set $\{\mu\}$  is  measurable.

\begin{definition}[Wasserstein barycenter]\label{def:barycenter}
	Let $(X,\tau, \ds)$ be an extended metric space and let $\bar \tau$ be a Hausdorff topology on  $\pr(X)$ so that $(\pr(X),  \bar \tau, W_2)$ is also an extended metric-topological space.    Let $\Omega\in  \pr_2(\pr(X), W_2)$ be a Radon probability  measure on $(\pr(X), \mathcal B(\bar \tau))$ with finite variance. We call $\bar{\nu}\in \pr(X)$ a  \emph{Wasserstein barycenter} of $\Omega$ if
	
	\begin{equation}
	\int_{\pr(X)} W_2^2(\mu,\bar{\nu})\,\d\Omega(\mu)=\min_{\nu\in \pr(X)}\int_{\pr(X)} W_2^2(\mu, \nu)\,\d \Omega(\mu).
	\end{equation}
\end{definition}

\medskip

We will also study a multi-marginal version of optimal transport, which is a nature generalization of classical  (two-marginal) optimal transport.  This problem  has been studied  in different forms and in different subjects  such as economics and statistics.  In particular,  Gangbo and Swiech\cite{gangbo1998optimal} studied the multi-marginal optimal transport  map in the Euclidean setting,  Agueh--Carlier \cite{AguehCarlier} studied Wasserstein barycenter  problem  via multi-marginal optimal transport,  we refer to \cite{PassSurvey} for a survey on this fast-developing topic.  

\begin{definition}[Multi-marginal optimal transport]\label{def:multi-marginal optimal transport}
	Let $\mu_1,...,\mu_n \in \pr (X)$ and let $c: X^n \to \R$ be a lower semi-continuous cost function. The multi-marginal optimal transport problem of Monge type is to study
	\begin{equation}\label{3.3}
		\inf_{ T_2,\dots,T_n} \int_{X} c(x_1,T_2(x_1),\dots,T_n(x_1)) \, \d\mu_1(x_1),
	\end{equation}
	where for each $i=2,\dots,n$, the map $T_i:X\to X$ pushes the measure $\mu_1$ forward to $\mu_i$,  that means, for any Borel $A\subseteq X$, $\mu_1(T_{i}^{-1}(A))=\mu_i(A)$.
	
The multi-marginal optimal transport problem of Kantorovich type is to study
\begin{equation}\label{3.4}
	\inf_{ \pi\in \Pi} \int_{X^n} c(x_1,\dots,x_n) \, \d\pi (x_1,\dots,x_n),
\end{equation}
where the infimum is taken over $\Pi(\mu_1,\dots,\mu_n)$,  the union of all probability measures $\pi$ on $X^n$ whose marginals are $\mu_1,\dots,\mu_n$ respectively .
\end{definition}

	By \cite[Theorem 4.1, Lemma 4.4]{V-O},  we can see that the set $\Pi(\mu_1,\dots,\mu_n)$ is tight.  So by Prohkhorov's theorem,   the multi-marginal optimal transport plan always exists:   there exists $\theta\in\Pi(\mu_1,\dots,\mu_n)$, such that 
	$$\int_{X^n} c(x_1,\dots,x_n) \, \d\theta=\min_{ \pi\in \Pi} \int_{X^n} c(x_1,\dots,x_n) \, \d\pi.$$

	It is worth to mention that if we choose $$c(x_1,\dots,x_n)=\inf_{y\in X}\sum_{i=1}^n \frac{1}{2} \ds^2(x_i,y),$$
	there is an intimate link between multi-marginal optimal transport and Wasserstein barycenter problem. This connection was studied by Gangbo--Swiech \cite{gangbo1998optimal}, Agueh--Carlier \cite{AguehCarlier} in the Euclidean setting,  by Kim--Pass \cite{kim2015multi, KimPassAIM} in  Riemannian manifolds and  by Jiang \cite{JiangWasserstein} in Alexandrov spaces. We will  further study this connection in the next two sections.

\subsection{Curvature-dimension condition}\label{SS:CDDef}

Throughout  this subsection, $\ms$ is a metric measure space, where $(X, \ds)$ is a separable complete geodesic space  and $\mm$ is a non-negative Radon measure with full support.

For any $t\in [0,1]$,  let ${\rm e}_{t}$ denote the evaluation map: 
$$
{\rm e}_{t} : \geo(X) \to X, \qquad \gamma\mapsto  \gamma_{t}.
$$
By super-position theorem \cite[Theorem 2.10]{AG-U},  any geodesic $(\mu_t)_{t \in [0,1]}$   in the Wasserstein space   $(\mathcal{P}_2(X, \ds), W_2)$  can be lifted to a measure $\nu \in {\mathcal {P}}(\geo(X))$, 
so that $({\rm e}_t)_\sharp \, \nu = \mu_t$ for all $t \in [0,1]$. 
Given $\mu_{0},\mu_{1} \in \mathcal{P}_{2}(X, \ds)$, we denote by 
$\Opt(\mu_{0},\mu_{1})$ the space of all $\nu \in \mathcal{P}(\geo(X))$ for which $({\rm e}_0,{\rm e}_1)_\sharp\, \nu$ 
realizes the minimum in the Kantorovich problem \eqref{eq:W2def}. Such  $\nu$ will be called { a dynamical optimal plan}. Since $(X,\ds)$ is geodesic,  the set  $\Opt(\mu_{0},\mu_{1})$ is non-empty for any $\mu_0,\mu_1\in \mathcal{P}_2(X, \ds)$.

\medskip

Next we introduce the curvature-dimension condition of non-smooth  metric measure space, which was introduced independently by Lott--Villani \cite{Lott-Villani09} and Sturm \cite{S-O1, S-O2}. Recall that the relative entropy  of $\mu\in\mathcal{P}_2(X, \ds)$ with respect to $\mm$ is defined by
\begin{equation*}
{\rm Ent}_\mm(\mu):=
\left \{\begin{array}{ll}
\int \rho\ln \rho\,\d \mm ~~~&\text{if}~ \mu\ll \mm,   \mu=\rho\,\mm\\
+\infty &\text{otherwise}.
\end{array}\right.
\end{equation*}

\begin{definition}[${\rm CD}(K,\infty)$ condition]\label{def:CD infty}
	Let $K \in \R$. A metric measure space  $(X,\ds,\mm)$ verifies ${\rm CD}(K,\infty)$ condition  if, for any two $\mu_{0},\mu_{1} \in \pr_{2}(X,\ds)$,  
	there exists $\nu \in \Opt(\mu_{0},\mu_{1})$, such that for all $t\in(0,1)$:
	\begin{equation}\label{eq:defCD infty}
	{\rm Ent}_{\mm}(\mu_t)\leq(1-t){\rm Ent}_{\mm}(\mu_0)+t{\rm Ent}_{\mm}(\mu_1)-\frac{K}{2}(1-t)tW_2^2(\mu_0,\mu_1),
	\end{equation}
	where $\mu_t:=(e_t)_{\sharp}\nu$.
\end{definition}

In order to formulate  the curvature-dimension condition with a dimensional parameter,  we need the  following distortion coefficients.  
For $\kappa \in \R$, define the functions $s_\kappa, c_\kappa: [0, +\infty) \to \R$ (on $[0, \pi/ \sqrt{\kappa})$ if $\kappa >0$) as:
\begin{equation} 
s_\kappa(\theta):=\left\{\begin{array}{lll}
(1/\sqrt {\kappa}) \sin (\sqrt \kappa \theta), &\text{if}~~ \kappa>0,\\
\theta, &\text{if}~~\kappa=0,\\
(1/\sqrt {-\kappa}) \sinh (\sqrt {-\kappa} \theta), &\text{if} ~~\kappa<0
\end{array}
\right.
\end{equation}
and
\begin{equation} 
c_\kappa(\theta):=\left\{\begin{array}{lll}
\cos (\sqrt \kappa \theta), &\text{if}~~ \kappa>0,\\
1, &\text{if}~~\kappa=0,\\
\cosh (\sqrt {-\kappa} \theta), &\text{if} ~~\kappa<0.
\end{array}
\right.
\end{equation}
It can be seen that $s_\kappa$ is a  solution to the following   ordinary differential equation
\begin{equation} \label{eq:riccati}
s_\kappa''+\kappa  s_\kappa=0.
\end{equation}

For $K\in \R, N\in [1,\infty), \theta \in (0,\infty), t\in [0,1]$, set 
we define the  distortion coefficients $\sigma_{K,N}^{(t)}$ and  $\tau_{K,N}^{(t)}(\theta)$ as
\begin{equation}\label{eq:Defsigma}
\sigma_{K,N}^{(t)}(\theta):= 
\begin{cases}
\infty, & \textrm{if}\ K\theta^{2} \geq N\pi^{2}, \crcr
t & \textrm{if}\ K \theta^{2}=0,  \crcr
\displaystyle   \frac{s_{\frac K{N}}(t\theta)}{s_{\frac K{N}}(\theta)} & \textrm{otherwise}
\end{cases}
\end{equation}
and
\begin{equation}\label{eq:deftau}
\tau_{K,N}^{(t)}(\theta): = t^{1/N} \sigma_{K,N-1}^{(t)}(\theta)^{(N-1)/N}.
\end{equation}

\medskip
The R\'enyi Entropy functional ${\mathcal E}_{N} : \mathcal P_2(X, \ds) \to [0, \infty]$ is defined by
\begin{equation*}
{\mathcal E}_{N}(\mu)  :=
\left \{\begin{array}{ll}
 \int \rho^{1-1/N}\,\d \mm ~~~&\text{if}~\mu\ll \mm,   \mu=\rho\,\mm\\
+\infty &\text{otherwise}.
\end{array}\right.
\end{equation*}

\begin{definition}[$\cdkn$ condition]\label{def:CD}
	Let $K \in \R$ and $N \in [1,\infty)$.   We say that a metric measure space  $(X,\ds,\mm)$ verifies $\cdkn$ condition,   if for any two $\mu_{0},\mu_{1} \in \pr_{2}(X,\ds)$ 
	with bounded support,   there exist $\nu \in \Opt(\mu_{0},\mu_{1})$ and  $\pi:=({\rm e}_0,{\rm e}_1)_\sharp\, \nu$,  such that $\mu_{t}:=(e_{t})_{\sharp}\nu \ll \mm$ and for any $N'\geq N, t\in [0,1]$:
	\begin{equation}\label{eq:defCD}
	{\mathcal E}_{N'}(\mu_{t}) \geq \int \tau_{K,N'}^{(1-t)} (\ds(x,y)) \rho_{0}^{-1/N'} 
	+ \tau_{K,N'}^{(t)} (\ds(x,y)) \rho_{1}^{-1/N'} \,\d\pi(x, y).
	\end{equation}
\end{definition}

\brk
It is worth recalling that if $(M,g)$ is a Riemannian manifold of dimension $n$ and 
$h \in C^{2}(M)$ with $h > 0$, then the metric measure space $(M,\ds_{g},h \, {\rm Vol}_{g})$ (where $\ds_{g}$ and ${\rm Vol}_{g}$ denote the Riemannian distance and volume induced by $g$) verifies $\cdkn$ with $N\geq n$ if and only if  (see  \cite[Theorem 1.7]{S-O2})
$$
\Ric_{g,h,N} \geq  K g, \qquad \Ric_{g,h,N} : =  \Ric_{g} - (N-n) \frac{\nabla_{g}^{2} h^{\frac{1}{N-n}}}{h^{\frac{1}{N-n}}}.  
$$
In particular if $N = n$ the generalized Ricci tensor $\Ric_{g,h,N}= \Ric_{g}$ makes sense only if $h$ is constant. 
\erk

\subsubsection*{Riemannian curvature-dimension condition}
  In \cite{AGS-M}, Ambrosio, Gigli and Savar{\' e} introduced a Riemannian variant of the curvature-dimension condition, called {\bf R}iemannian {\bf C}urvature-{\bf D}imension condition,  ruling out Finsler manifolds.

Let us first recall the definition of the Cheeger energy \cite{C-D}. The Cheeger energy is a functional defined in $ L^2(X, \mm)$ by
$$
{\rm Ch}(f):=\inf\left\{\liminf_{i\rightarrow\infty}\int_X \left|\lip {f_i}\right|^2\,\d\mm:f_i\in {\rm Lip}(X, \d),\left\|f_i-f\right\|_{L^2}\rightarrow 0\right\},
$$
where
$$
\left|\lip h\right|(x):=\limsup_{y\rightarrow x}\frac{\left|h(y)-h(x)\right|}{\ds(x,y)}
$$
for $h\in \Lip(X,\ds)$.  The Sobolev space  $W^{1,2}\ms$ is defined as
$$
W^{1,2}\ms:=\left\{f\in L^2(X, \mm):  {\rm Ch}(f)<\infty\right\}.
$$
\begin{definition}[${\rm RCD}(K,\infty)$ condition]
	Let $K \in \R$.  We say that a metric measure space  $(X,\ds,\mm)$ satisfies  Riemannian curvature-dimension condition $\rcd$ (or $(X,\ds,\mm)$ is an ${\rm RCD}(K,\infty)$ space for simplicity) if it satisfies ${\rm CD}(K,\infty)$ condition and the Cheeger energy ${\rm Ch}$ is a quadratic form in the sense that
	\begin{equation}
	{\rm Ch}(f+g)+{\rm Ch}(f-g)=2{\rm Ch}(f)+2{\rm Ch}(g)\qquad \forall~f,g\in W^{1,2}\ms.
	\end{equation}
\end{definition}

Concerning the finite dimensional case, there is  a Riemannian  curvature-dimension condition $\text{RCD}(K,N)$ (cf. \cite{G-O, AMS-N, EKS-O}):

\begin{definition}[${\rm RCD}(K,N)$ condition]
		Let $K \in \R$ and $N\in (0,\infty)$.   We say that a metric measure space  $(X,\ds,\mm)$ satisfies ${\rm RCD}(K,N)$ condition,   if it satisfies ${\rm CD}(K,N)$ and the Cheeger energy ${\rm Ch}$ is a quadratic form.
\end{definition}

\begin{remark}
For any  Riemannian manifold $(M^n, g)$ and any density function  $h$,  the Sobolev space   $W^{1,2}(M,\ds_{g},h \, {\rm Vol}_{g})$ is  a Hilbert space.  So  $(M,\ds_{g},h \, {\rm Vol}_{g})$ satisfies $\rcdkn$ if and only if it is a $\cdkn$ space.
\end{remark}
\begin{remark}
A variant of the $\cdkn$ condition, called  reduced curvature dimension condition and denoted by  $\CD^{*}(K,N)$ \cite{BS-L},  asks for the same inequality \eqref{eq:defCD} of $\cdkn$ but  the
coefficients $\tau_{K,N}^{(t)}(\ds(\gamma_{0},\gamma_{1}))$ and $\tau_{K,N}^{(1-t)}(\ds(\gamma_{0},\gamma_{1}))$ 
are replaced by $\sigma_{K,N}^{(t)}(\ds(\gamma_{0},\gamma_{1}))$ and $\sigma_{K,N}^{(1-t)}(\ds(\gamma_{0},\gamma_{1}))$ respectively.
For both definitions there is a local version and it was recently proved in \cite{CavallettiEMilman-LocalToGlobal} that  on  an essentially  non-branching metric measure spaces with $\mm(X)<\infty$ (and in \cite{LZH22}  for general $\mm$), the  $\CD^{*}_{loc}(K,N)$, $\CD^{*}(K,N)$, $\CD_{loc}(K,N)$, $\cdkn$ conditions are all equivalent, for all $K\in \R, N\in (1,\infty)$, 
via the needle decomposition method.   In particular,  ${\rm RCD}^{*}(K,N)$ and ${\rm RCD}(K,N)$ are equivalent.   
\end{remark}

\medskip

\begin{example}[Notable examples of RCD spaces]
The class of ${\rm RCD}(K,N)$ spaces includes the following remarkable subclasses:
\begin{itemize}
\item Measured Gromov--Hausdorff limits of  $N$-dimensional Riemannian manifolds  with  $\mathrm {Ricci} \geq K$,  see \cite{AGS-M, GMS-C}; 
\item $N$-dimensional Alexandrov spaces with curvature bounded from below by $K$, see  \cite{ZhangZhu10, Petrunin11};
\item Cones, spherical suspensions, Warped products over {\rm RCD} space, see  \cite{K-C, K-R}.
\end{itemize}
We refer the readers to Villani's Bourbaki seminar \cite{VillaniJapan}, Ambrosio's  ICM lecture \cite{AmbrosioICM}  and Sturm's ECM lecture \cite{SturmEMS}  for  an overview of this fast-growing field and bibliography.
\end{example}

\subsection{Gradient Flows: EVI type}
We recall some basic results about gradient flow theory on an extended metric space $(X,\ds)$, which play key roles in our proof of Jensen's inequality. We refer the readers to \cite{ambrosio2005gradient, ambrosio2016optimal}   for  comprehensive discussion and references about this topic.

\begin{definition}[${\rm EVI}_{K}$ formulation of gradient flows]\label{def:EVI}
	Let $(X,\ds)$ be an extended metric space,  $K\in\mathbb{R}$, $E:X\rightarrow \mathbb{R}\cup \left\{+\infty\right\}$ be a lower semi-continuous function.   For any $y_0\in \De(E):={\left\{E<+\infty\right\}}$,  we say that $(0,\infty)\ni t\rightarrow y_t \in X$ is an ${\rm EVI}_{K}$ gradient flow  of $E$, starting from $y_0$,  if it is a locally absolutely continuous curve,  such that $y_t\overset{\ds}{\rightarrow}y_0$ as $t\rightarrow 0$ and the following inequality is satisfied: for any $z\in \De (E)$ satisfying $\ds(z, y_t)<\infty$ for some (and then all) $t\in (0, \infty)$:
	\begin{equation}
	\frac{1}{2}\frac{\d }{\d t}\ds^2(y_t,z)+\frac{K}{2}\ds^2(y_t,z)\leq	E(z)-E(y_t),\qquad \text{for a.e. } t\in(0,\infty).
	\end{equation}
	If $x_0\notin \De (E)$ we 
	ask for 
	\[
\liminf_{t\downarrow 0} F(x_t) \geq F(x_0),~~ \lim_{t\downarrow 0} \ds(x_t, y) \to \ds(x_0, y),~~ ~~\forall y\in  \overline{  {\rm D}(F)}.
 \]
\end{definition}

 The existence of ${\rm EVI}_K$ gradient flow of $K$-convex functionals has been studied in many cases, e.g. smooth Riemannian manifolds,  infinite-dimensional Hilbert spaces, ${\rm CAT}(k)$ spaces, Wasserstein spaces over Riemannian manifolds. We refer to \cite{muratori2020gradient, V-O, sturm2003probability} for more discussions.

In particular,  Ambrosio--Gigli--Savar\'e \cite{AGS-M} showed that Riemannian  curvature-dimension conditions  $\rcd$ can be characterized via ${\rm EVI}_K$ gradient flows of the relative entropy.
\begin{theorem}[\cite{AGS-M}, Theorem 5.1; \cite{AGMR-R}, Theorem 6.1]\label{rcdevi}
	Let $(X,\ds,\mm)$ be a metric measure space and $K\in\mathbb{R}$. Then $(X,\ds,\mm)$ satisfies ${\rm RCD}(K,\infty)$  if only and if  $(X,\ds,\mm)$ is a length space satisfying the exponential growth condition
	$$
	\int_X e^{-c\ds(x_0,x)^2}\,\d\mm(x)<\infty,\qquad \text{for some } x_0 \in X \text{ and } c>0,
	$$
	and for all $\mu\in\mathcal{P}_2(X, \ds)$ there exists an ${\rm EVI}_K$ gradient flow, in $(\mathcal{P}_2(X, \ds),W_2)$, of ${\rm Ent}_{\mm}$ starting from $\mu$.
\end{theorem}

\medskip
Concerning abstract Wiener spaces, which are typical extended metric measure spaces,  Ambrosio--Erbar--Savar\'e  \cite{ambrosio2016optimal} proved the following theorem (see also an earlier paper by Fang--Shao--Sturm\cite{FSS10}).

\begin{theorem}[\cite{ambrosio2016optimal}, Theorem 11.1, \S11 and \S 13]\label{wnevi}
	Let $(X,H,\gamma)$ be an abstract Wiener space equipped with the canonical Cameron--Martin distance  $\ds_H$ and the Gaussian measure $\gamma$.  Then  for all $\mu\in\mathcal{P}(X)$ with $\mu\ll \gamma$,  there exists an ${\rm EVI}_1$ gradient flow of ${\rm Ent}_{\gamma}$ starting from $\mu$.
\end{theorem}
\medskip

Concerning the finite dimensional case, similarly, metric measure spaces with Riemannian Ricci curvature bounded from below and with   dimension bounded from above can also be characterized via  ${\rm EVI}_{K, N}$ gradient flows:

\begin{definition}[${\rm EVI}_{K,N}$ formulation of gradient flows]\label{def:EVI K,N}
	Let $(X,\ds)$ be a metric space, $E:X\rightarrow \mathbb{R}\cup \left\{+\infty\right\}$ be a lower semi-continuous functional, $y_0\in \overline{\left\{E<\infty\right\}}$, and $K\in\mathbb{R}$, $N\in(0,\infty)$. We say that $(0,\infty)\ni t\rightarrow y_t \in X$ is an ${\rm EVI}_{K,N}$ gradient flow of $E$, starting from $y_0$,  if it is a locally absolutely continuous curve,  such that $y_t\overset{\ds}{\rightarrow}y_0$ as $t\rightarrow 0$ and the following inequality is satisfied for any $z\in X$
	\begin{equation}
		\frac{\d}{\dt}s^2_{K/N}\left(\frac{1}{2}\ds(y_t,z)\right)+K s^2_{K/N}\left(\frac{1}{2}\ds(y_t,z)\right)\leq	\frac{N}{2}\left(1-\frac{U_N(z)}{U_N(y_t)}\right),
	\end{equation}
	for a.e. $t\in(0,\infty)$, where $U_N(x):=e^{-E(x)/N}$.
\end{definition} 

\medskip

\begin{theorem}[\cite{EKS-O}, Theorem 3.17]\label{th:rcdknevi}
	Let $(X,\ds,\mm)$ be a metric measure space. Let $K\in\mathbb{R}$ and $N\in(0,\infty)$. Then $(X,\ds,\mm)$ is an ${\rm RCD}(K,N)$ space if only and if $(X,\ds,\mm)$ is a length space satisfying the exponential growth condition
	$$
	\int_X e^{-c\ds(x_0,x)^2}\,\d\mm(x)<\infty,\qquad \text{for some}~ x_0 \in X \text{ and } c>0,
	$$
	and for all $\mu\in\mathcal{P}_2(X, \ds)$ there exists an ${\rm EVI}_{K,N}$ gradient flow, in $(\mathcal{P}_2(X, \ds),W_2)$, of ${\rm Ent}_{\mm}$ starting from $\mu$.
\end{theorem}

\section{Wasserstein barycenter: general setting}\label{section:existence}
\subsection{Existence}

The study of the Wasserstein barycenter problem was initiated   by Agueh and Carlier \cite{AguehCarlier}, who established the existence,   uniqueness and regularity results for finitely supported measure $\Omega\in \pr_2(\pr_2(\R^n, |\cdot |), W_2)$.   Later on,  these results have been extended to more spaces including:  Kim--Pass \cite{KimPassAIM} in the Riemannian setting and Jiang \cite{JiangWasserstein} in the Alexandrov spaces. 

Concerning the existence of the Wasserstein barycenter,  we have the following existence result proved by Le Gouic and Loubes \cite{le2017existence}. Note that,  in general, the  Wasserstein space is neither locally compact nor NPC, even if the underling space is  locally compact and NPC,  so  the existence is not obvious at all (cf. Example \ref{example1}).

\begin{proposition}[\cite{le2017existence},Theorem 2]\label{proper space}
	Let $(X,\ds)$ be a separable,  locally compact,  geodesic space.  Then  any $\Omega\in \pr_2(\pr_2(X, \ds), W_2)$ has a (Wasserstein) barycenter.
\end{proposition}

To the best of our knowledge, the proposition above is the best published result concerning the existence of the Wasserstein barycenter.  However,  the local compactness condition is not valid in  infinite-dimensional spaces, such as infinite-dimensional Hilbert space and general ${\rm RCD}(K,\infty)$ spaces.  By adapting the proof of  \cite[Proposition 4.2]{AguehCarlier} (or  \cite[Theorem 8]{le2017existence})
we can prove an existence theorem for \emph{barycenter extended metric spaces} (recall Definition \ref{def:bcspace}),  without  any compactness condition.

We begin with a lemma concerning the existence of measurable barycenter-selection map.

\begin{lemma}\label{measurable selection}
	Let $(X,\ds)$ be a barycenter space and $\{\lambda_{i}\}_{i=1,..., n}$ be  positive real numbers with  $\sum_{i=1}^{n} \lambda_{i}=1$.  There exists a measurable barycenter selection map $T: X^{n} \to X$, such that $T\left(x_{1}, \ldots, x_{n}\right)$ is a barycenter of  $\sum_{i=1}^{n} \lambda_{i} \delta_{x_{i}} \in \pr_0(X) \cap \pr_2(X, \ds).$
\end{lemma}

\begin{proof}
	The proof is similar to the  geodesic case,   see \cite[Theorem 2.10 and Lemma 2.11]{AG-U}. Since $(X,\ds)$ is a barycenter space,   any $\sum_{i=1}^{n} \lambda_{i} \delta_{x_{i}} \in \pr_0(X) \cap \pr_2(X, \ds)$  has a barycenter.   So we can define a multivalued map  $G: X^n \to X$, which associates to each  $(x_1,\dots,x_n)\in X^n$ a set $G(x_1,\dots,x_n)$ consisting of all  barycenters of  $\sum_{i=1}^{n} \lambda_{i} \delta_{x_{i}}$,   provided $\sum_{i=1}^{n} \lambda_{i} \delta_{x_{i}}\in \pr_2(X, \ds)$; otherwise,  we define $G(x_1,\dots,x_n)=y$ for a fixed point $y\in X$. We can see that $G$    has closed graph. Then using Kuratowski--Ryll--Nardzewski measurable selection theorem \cite[Theorem 6.9.3]{bogachev2007measure}, we can find a measurable barycenter selection map $T$.
\end{proof}

Next we prove an existence  theorem concerning the Wasserstein barycenter problem in barycenter spaces. As a corollary of this theorem, we know the Wasserstein spaces over an abstract Wiener space is barycentric.

\begin{theorem}[Wasserstein space over a barycenter space is barycentric]\label{barycenter space}
	Let $(X,\ds)$ be a barycenter space, then $\mathcal W_2=(\pr(X),W_2)$ is barycentric as well.
\end{theorem}

\begin{proof}
By Lemma \ref{measurable selection}, there exists a measurable barycenter selection map $T: X^{n} \to X$,  such that for any $x_1,\dots,x_n\in X$ and ${\bf x}:=(x_1,\dots,x_n)$,  there is $T({\bf x})=T(x_1,\dots,x_n)\in X$ satisfying
	\begin{equation}\label{eq1:thbc1}
		c(x_1,\dots,x_n) :=\inf_{y\in X}\sum_{i=1}^n \lambda_i \ds^2(x_i,y)=\sum_{i=1}^n \lambda_i \ds^2(x_i,T({\bf x}))\in [0, +\infty].
	\end{equation}

Let $\mu_1,\dots,\mu_n\in \pr(X)$,  $\{\lambda_{i}\}_{i=1,..., n}$ be  positive real numbers with  $\sum_{i=1}^{n} \lambda_{i}=1$.  Assume  $\Omega:=\sum_{i=1}^{n} \lambda_{i} \delta_{\mu_{i}}\in \pr_2(\pr(X), W_2)$. We claim $\nu=T_{\sharp} \pi$ is a Wasserstein barycenter of $\Omega$, where $\pi\in\Pi(\mu_1,\dots,\mu_n)$ is a solution of the multi-marginal optimal transport problem  of Kantorovich type \eqref{3.4}.
	
	\paragraph*{On one hand:}  by  disintegration theorem \cite[Theorem 5.3.1]{ambrosio2005gradient}, for  any  $\eta_i\in\Pi(\mu_i,\mu)$, $i=1,\dots,n$, we can write $\eta_i=\eta_i^y\otimes\mu$ for a Borel family of probability measures ${\{\eta_i^y\}}_{y\in X}\subseteq \pr(X)$. 
	This means, for any Borel-measurable function $f:X^{2}\to [0,+\infty)$,
  \begin{equation}\label{eq:dis}
  \int_{X\times X } f\, \d\eta_i=\int_{X}\int_{X}f\,\d\eta_i^y\d\mu(y).
  \end{equation}
Denote $\eta_y=\eta_1^y\dots\eta_n^y$ and $\eta:=\eta_y\mu\in \pr(X^{n+1})$. It can be seen that $\eta\in\Pi(\mu_1,\dots,\mu_n,\mu) $ and  for any Borel-measurable function $f:X^{n+1}\to [0,+\infty)$, we have
	\begin{equation}\label{eq2:thbc1}
		\int_{X^{n+1} } f(x_1,...,x_n, y)\, \d\eta(x_1,...,x_n, y)=\int_{X}\int_{X^n}f(x_1,...,x_n, y)\,\d\eta_y(x_1,...,x_n)\d\mu(y).
	\end{equation}
	
For any $\mu\in \pr(X)$ with
	\[
	\sum_{i=1}^n \lambda_i W_2^2(\mu_i,\mu)<+\infty,
	\]
there are  optimal transport plans $\{\eta_i\}_{i=1,2,...n}$,   so that
	\begin{equation}\label{5.5}
		\begin{aligned} 
		&\sum_{i=1}^n \lambda_i W_2^2(\mu_i,\mu)\\
			&=\sum_{i=1}^n \lambda_i \int_{X\times X} \ds^2(x_i,y) \, \d\eta_i
			\overset{\eqref{eq:dis}}=\int_{X^{n+1}}\sum_{i=1}^n \lambda_i \ds^2(x_i,y) \, \d\eta \\
			&\overset{\eqref{eq1:thbc1}}\geq \int_{X^{n+1}}\sum_{i=1}^n \lambda_i \ds^2(x_i,T({\bf x})) \, \d\eta\\
			&\overset{\eqref{eq2:thbc1}}=\int_{X}\int_{X^n}\sum_{i=1}^n \lambda_i \ds^2(x_i,T({\bf x}))\,\d\eta_y({\bf x})\,\d\mu(y)\\
			&\overset{\text {Fubini}}=\int_{X^n}\sum_{i=1}^n \lambda_i \ds^2(x_i,T({\bf x}))\,\d\tilde{\eta}\\
			&\geq \int_{X^n}\sum_{i=1}^n \lambda_i \ds^2(x_i,T({\bf x}))\,\d\pi\overset{\text{(4.1)}}=\int_{X^n}c(x_1,\dots,x_n)\,\d\pi,
		\end{aligned}
	\end{equation}
	where we  use the facts $\tilde{\eta}:=\int_X \eta_y\,\d\mu(y)\in \Pi(\mu_1,\dots,\mu_n) $ in the last inequality. 
In particular,   as $\sum \lambda_i \delta_{\mu_{i}}$ has finite variance,  there is $\mu\in \pr(X)$ so that $\sum_{i=1}^n \lambda_i W_2^2(\mu_i,\mu)<+\infty$,  we have
\begin{equation}\label{eq:mulfi}
\int_{X^n}c(x_1,\dots,x_n)\,\d\pi<+\infty.
\end{equation}
	
\paragraph*{On the other hand:}  for any $i=1,\dots,n$, denote by  $\theta_i$ the $i$-th canonical projection from $X^n$ to $X$, and denote $\eta_i=(\theta_i,T)_{\sharp}\pi$. By construction, $\eta_i\in\Pi(\mu_i,\nu)$.
	Then by  definition of  $L^{2}$-Kantorovich--Wasserstein distance, we have
	\begin{equation}
		W_2^2(\mu_i,\nu)\leq  \int_{X\times X} \ds^2(x_i,y) \, \d\eta_i
		=\int_{X^n} \ds^2(x_i,T({\bf x})) \, \d\pi.
	\end{equation}
	Combining with \eqref{eq1:thbc1} and \eqref{eq:mulfi} we get
	\begin{equation}\label{5.3}
		\begin{aligned}
			\sum_{i=1}^n \lambda_i W_2^2(\mu_i,\nu)\leq \sum_{i=1}^n \lambda_i \int_{X^n} \ds^2(x_i,T({\bf x})) \, \d\pi=\int_{X^n}c(x_1,\dots,x_n)\, \d\pi<+\infty. 		
		\end{aligned}
	\end{equation}
	
Combining  \eqref{5.5} and \eqref{5.3}, we  get
	\begin{equation}
		\sum_{i=1}^n \lambda_i W_2^2(\mu_i,\nu)=\inf_{\mu}\sum_{i=1}^n \lambda_i W_2^2(\mu_i,\mu)=\int_{X^n}\sum_{i=1}^n \lambda_i \ds^2(x_i,T({\mathbf x}))\,\d\pi.
	\end{equation}
	
	Therefore, ${\nu}=T_{\sharp} \pi$ is a Wasserstein barycenter of $\Omega$ and we prove the claim.
	\end{proof}
	
\bigskip

 Furthermore,  we have the following super-position theorem about Wasserstein barycenters.  We refer the readers to \cite[Theorem 2.10]{AG-U} for a well-known super-position theorem concerning Wasserstein geodesics.
	
	\begin{theorem}[Super-position theorem]\label{th:sp}
	Let $(X, \ds)$ be an extended metric space. Let $\mu_1,\dots,\mu_n\in \pr(X)$,  $\{\lambda_{i}\}_{i=1,..., n}$ be  positive real numbers with  $\sum_{i=1}^{n} \lambda_{i}=1$.   Assume    $\Omega:=\sum_{i=1}^{n} \lambda_{i} \delta_{\mu_{i}}$ has bounded variance and has a barycenter $\mu \in \pr(X)$.  
	
	Then  there is $\eta \in \Pi(\mu_0,\dots, \mu_n, \mu)$ so that  for $\eta$-a.e. $(x_1, \dots, x_n, y)\in X^{n+1}$,   $y$ is a barycenter of $\sum_i \delta_{x_i}$.
	\end{theorem}
\begin{proof}
First of all,  we claim:
	\begin{equation}\label{eq:mmg1}
		\inf_{\nu\in \pr(X)}\sum_{i=1}^n \lambda_i W_2^2(\mu_i,\nu)\leq \inf_{\pi\in\Pi}\int_{X^n}c(x_1,\dots,x_n)\,\d\pi.
	\end{equation}
	
	Given $\epsilon>0$, similar to Lemma \ref{measurable selection},  by measurable  selection theorem we can find a measurable map $T_\epsilon:X^n\rightarrow X$, such that  for any ${\bf x}:=(x_1,\dots,x_n)$, it holds
	$$c({\bf x}):=\inf_{y\in X}\sum_{i=1}^n \lambda_i \ds^2(x_i,y)>\sum_{i=1}^n \lambda_i\ds^2(x_i,T_\epsilon({\mathbf x}))-\epsilon.$$	
Similar to the proof of Theorem \ref{barycenter space},  for  the $i$-th canonical projection  $\theta_i$ from $X^n$ to $X$, and any $\pi\in\Pi(\mu_1,\dots,\mu_n)$,  we   set $\nu_\epsilon:=(T_\epsilon)_\sharp \pi$ and  $\eta_i:=(\theta_i,T_\epsilon)_{\sharp}\pi\in\Pi(\mu_i,\nu_\epsilon)$.  Then
	$$
W_2^2(\mu_i,\nu_\epsilon)\leq  \int_{X\times X} \ds^2(x_i,y) \, \d\eta_i = \int_{X^n} \ds^2(x_i,T_\epsilon({\mathbf x})) \, \d\pi.
	$$
	Thus,
	\begin{equation}
	\inf_{\nu\in \pr(X)}\sum_{i=1}^n 	\lambda_i W_2^2(\mu_i,\nu)\leq \sum_{i=1}^n \int_{X^n} \lambda_i\d^2(x_i,T_\epsilon({\mathbf x})) \, \d\pi<\int_{X^n}c({\mathbf x})\, \d\pi+\epsilon.
	\end{equation}
Letting $\epsilon \to 0$, we prove the claim.

Let $\mu$ be a barycenter of $\Omega$.  By definition we have ${\rm Var}(\Omega)=\sum_{i=1}^n \lambda_i W_2^2(\mu_i,\mu)<\infty$.  Similar to  \eqref{5.5} in  the proof of Theorem \ref{barycenter space},  we can find two couplings $\tilde \eta \in \Pi(\mu_1,\dots,\mu_n)$ and $\eta \in  \Pi(\mu_1,\dots,\mu_n, \mu)$ such that 
\begin{eqnarray*}
\int_{X^n}c({\mathbf x})\, \d\tilde\eta&=&\int_{X^{n+1}}\inf_{y\in X}\sum_{i=1}^n \lambda_i \ds^2(x_i,y)\,\d \eta \leq \int_{X^{n+1}}\sum_{i=1}^n \lambda_i \ds^2(x_i,y) \, \d\eta\\&=& \sum_{i=1}^n \lambda_i W_2^2(\mu_i,\mu)=\inf_{\nu\in \pr(X)}\sum_{i=1}^n \lambda_i W_2^2(\mu_i,\nu).
\end{eqnarray*}
Combining with \eqref{eq:mmg1}, we know 
\[
\int_{X^{n+1}}\inf_{y\in X}\sum_{i=1}^n \lambda_i \ds^2(x_i,y)\,\d \eta= \int_{X^{n+1}}\sum_{i=1}^n \lambda_i \ds^2(x_i,y) \, \d\eta.
\]
In particular, for $\eta$-a.e. $(x_1,\dots, x_n, y)\in X^{n+1}$,  $y$ is a barycenter of $\sum_i \delta_{x_i}$.
\end{proof}

	\subsection{Uniqueness}\label{sec:uniq}
	
 By \cite[Theorem 3.1]{KimPassAIM} and \cite[Lemma 3.2.1]{pass2013optimal}, the uniqueness of Wasserstein barycenter can be deduced from the strict convexity of $L^{2}$-Kantorovich-Wasserstein distance with respect to the linear interpolation.  Note that the strict convexity can be proved using the existence of optimal transport maps. Thus the uniqueness is true for more general spaces, such as non-branching ${\rm CD}(K,N)$ spaces \cite[Theorem 3.3]{Gigli12a}, essentially non-branching ${\rm MCP}(K,N)$ spaces \cite[Theorem 1.1]{cavalletti2017optimal} and  abstract Wiener spaces \cite[Theorem 6.1]{FU}.   
		In this section,  to make our paper complete and self-contained,  we give a sketched proof to the uniqueness of the Wasserstein barycenter,  under the existence of optimal transport map. 
	
	Given an extended metric measure space $\ms$.  For $\Omega\in \pr_2(\mathcal W_2)$, we always assume  that the set of $\mm$-absolutely continuous probability measures $\pr_{ac}\ms \subset \pr(X)$ is   $\Omega$ measurable.

	\begin{proposition}\label{uniqueness in finite space}
	Let $(X,\ds,\mm)$ be an extended  metric measure space, such that the optimal transport problem of Monge type is solvable  if one of the marginal measures is in  $\pr_{ac}\ms$ (in this case we call $\ms$ a Monge Space, see  E.  Milman's paper \cite[\S 3.3]{milman2020quasi} for more discussions). Let $\Omega\in \pr_2(\mathcal W_2)$ be such that $\Omega(\pr_{ac}\ms)>0$, then the   barycenter of $\Omega$ is unique.
\end{proposition}

\begin{proof}
 From the definition of $L^{2}$-Kantorovich--Wasserstein distance, we can see that $\pr(X)\ni \nu \mapsto W_2^2(\mu, \nu)$ is convex (w.r.t. linear interpolation) for any $\mu\in \pr(X)$,  i.e.
		\begin{equation}\label{5.18}
		W_2^2(\mu, \lambda\nu_1+(1-\lambda)\nu_2)\leq \lambda W_2^2(\mu, \nu_1)+(1-\lambda)W_2^2(\mu, \nu_2).
	\end{equation}
	We claim that  this inequality  is strict if $\mu\in \pr_{ac}\ms$.
	
	In order to prove this claim, it is sufficiently to show that for any $\nu_1,\nu_2\in \pr(X)$, $\nu_1\neq \nu_2$, $0<\lambda<1$, it holds 
\[ W_2^2(\mu, \lambda\nu_1+(1-\lambda)\nu_2)< \lambda W_2^2(\mu, \nu_1)+(1-\lambda)W_2^2(\mu, \nu_2)
	\]
	for any  $\mu\in \pr_{ac}\ms$ with $W_2^2(\mu, \nu_1), W_2^2(\mu, \nu_2)<\infty$.

	Indeed, if \eqref{5.18} is an equality for some $0<\lambda<1$, denote by $\nu=\lambda\nu_1+(1-\lambda)\nu_2\in\pr(X)$. Since $\mu\in \pr_{ac}\ms$, by assumption, there exist  (unique) optimal transport maps $T_1$ and $T_2$, push forward $\mu$ to $\nu_1,\nu_2$ respectively.  Let $\gamma=\lambda T_1+(1-\lambda)T_2$, and note that \eqref{5.18} is an equality implies $\gamma$ is an optimal plan between $\mu$ and $\nu$, then $\gamma$ must induced by a map $T$, this contradicts with $\nu_1\neq \nu_2$. 
	
	Therefore, integrating $\nu\to W_2^2(\mu, \nu)$ with respect to $\Omega$ yields the strict convexity of the functional $\nu\to \int_{\mathcal W_2}W_2^2(\mu, \nu)\,\d\Omega(\mu)$ under the assumption $\Omega(\pr_{ac}\ms)>0$. This then implies  the uniqueness of its minimizer, the Wasserstein barycenter of $\Omega$.
\end{proof}

\section{Functionals on the Wasserstein space}\label{main}
\subsection{${\rm EVI}_K$ gradient flows}
\subsubsection{EVI implies Jensen's inequality}
Motivated  by a result of Daneri--Savar\'e \cite[Theorem 3.2]{daneri2008eulerian}, which tells us that a functional $E$ is $K$-convex along \emph{any} geodesic contained in $\overline {\left\{E<\infty \right\}}$ if $E$ has an ${\rm EVI}_K$ gradient flow from any starting point in $\overline {\left\{E<\infty \right\}}$,  we realize that the existence of ${\rm EVI}_K$ gradient flows of $E$ is closely related to  Jensen's inequality .

\medskip

The following proposition,   proved in \cite[Proposition 3.1]{daneri2008eulerian} (see also \cite{ambrosio2016optimal}),  provides an integral characterization of ${\rm EVI}_{K}$  gradient flows. We will use this formulation to prove Jensen's inequality.

\begin{proposition}[Integral version of $\rm EVI$]\label{integral version}
	Let $E$, $K$ and $(y_t)_{t>0}$ be as in Definition \ref{def:EVI}, then $(y_t)_{t>0}$ is an ${\rm EVI}_{K}$ gradient flow if and only if it satisfies 
	\begin{equation}\label{eq:integral version}
	\frac{e^{K(t-s)}}{2}\ds^2(y_t,z)-\frac{1}{2}\ds^2(y_s,z)\leq I_K(t-s)\big(E(z)-E(y_t)\big),\quad \forall  0\leq s\leq t,
	\end{equation}
	where $I_K(t)$:=$\int_{0}^{t}e^{Kr}\,\d r$.
\end{proposition}

\begin{theorem}[$\rm EVI$ implies Jensen's inequality]\label{EVI JI}
	Let $(X,\ds)$ be an extended metric space,  $K\in \R$,  $E$ be as in Definition \ref{def:EVI} and $\mu$ be a  probability measure  over $X$ with finite variance.   Let $\epsilon\geq 0.$ Assume there  is an ${\rm EVI}_K$ gradient flow  $(y_t)_{t>0}$ of $E$  starting from $y\in X$,   such that
	\begin{equation}\label{nearly barycenter}
		\int_X \ds^2(y,z)\,\d\mu(z)\leq {\rm Var}(\mu)+\epsilon.
	\end{equation}
	Then the following inequality holds:
	
	\begin{equation}\label{NearlyJI}
	E(y_t)\leq \int_{X} E(z)\,\d\mu(z)-\frac{K}{2}{\rm Var}(\mu)+\frac{\epsilon}{2I_K (t)}.
	\end{equation}
	
In particular, if   $\mu$ has a barycenter $\bar{y}$ and there is  an ${\rm EVI}_K$ gradient flow of $E$ starting from $\bar y$, then  we have Jensen's inequality:
	\begin{equation}\label{JI}
	E(\bar{y})\leq \int_{X} E(z)\,\d\mu(z)-\frac{K}{2}\int_{X}\ds^2(\bar{y},z)\,\d\mu(z).
	\end{equation}
\end{theorem}

\begin{proof}
	Integrating \eqref{eq:integral version} in $z$ with $\mu$ and choosing $s=0$,  we get
	\begin{equation}\label{EVI JI1}
	\frac{e^{Kt}}{2}\int_{X}\ds^2(y_t,z)\,\d\mu(z)-\frac{1}{2}\int_{X}\ds^2(y,z)\,\d\mu(z)\leq I_K(t)\left(\int_{X} E(z)\,\d\mu(z)-E(y_t)\right).
	\end{equation}
	  By Definition \ref{def:barycenter}, we have  $\int_{X}\ds^2(y_t,z)\,\d\mu(z)\geq{\rm Var}(\mu)$. Combining with  \eqref{nearly barycenter} and \eqref{EVI JI1}, we get
	\begin{equation}\label{EVI JI2}
	\frac{e^{Kt}-1}{2}{\rm Var}(\mu)-\frac{\epsilon}{2}\leq I_K(t)\left(\int_{X} E(z)\,\d\mu(z)-E(y_t)\right).
	\end{equation} 
	Dividing both sides of \eqref{EVI JI2} by $I_K(t)$, we get \eqref{NearlyJI}. 
	
Assume $\bar y$ is a barycenter of $\mu$ and $(y_t)_{t>0}$ is an ${\rm EVI}_K$ gradient flow from $\bar y$.   Taking $\epsilon=0$ in \eqref{NearlyJI} we get
	$$
	E(y_t)\leq \int_{X} E(z)\,\d\mu(z)-\frac{K}{2}{\rm Var}(\mu).
	$$
	Letting $t\rightarrow 0$,  by lower semi-continuity of $E$ and the fact $y_t\overset{\ds}{\rightarrow} \bar y$, we get Jensen's inequality \eqref{JI}.
\end{proof}

\subsubsection{Existence and uniqueness}\label{5.1.2}
Using Theorem \ref{EVI JI}, we  can extend Wasserstein Jensen's inequality (for the Wasserstein barycenter) on Riemannian manifolds (cf. \cite{KimPassAIM} ) to some non-smooth extended metric measure spaces,  including  $\rcd$ spaces and abstract Wiener spaces. Unlike the known approaches,  this approach (using  $\rm EVI$ gradient flow) does not rely on absolute continuity of the barycenter or  local compactness of the space.  In fact,   we can  deduce the  absolute continuity of Wasserstein barycenter(s),  as a corollary
of this Wasserstein Jensen's inequality.

\begin{theorem}[Wasserstein Jensen's inequality]\label{WJI}
Let $K\in \mathbb{R}$ and let $(X,\ds,\mm)$ be an extended metric measure space,  $\Omega$ be a  probability measure over $\mathcal{P}(X)$ with finite variance. 

Assume   that  any  ${\mu} \in \pr(X)$ which has finite distance to $\De( {\rm Ent}_{\mm})$ is the starting point of an ${\rm EVI}_K$ gradient flow of ${\rm Ent}_{\mm}$ in the Wasserstein space.
Then the following Wasserstein Jensen's inequality is valid for any barycenter $\bar{\mu}$ of $\Omega$:
	\begin{equation}\label{WJI1}
	{\rm Ent}_{\mm}(\bar{\mu})\leq \int_{\mathcal{P}(X)}{\rm Ent}_{\mm}(\mu)\,\d\Omega(\mu)-\frac{K}{2}\int_{\mathcal{P}(X)}W_2^2(\bar{\mu},\mu)\,\d\Omega(\mu).
	\end{equation}
As a consequence,   suppose  that 
	$$
	\int_{\mathcal{P}(X)}{\rm Ent}_{\mm}(\mu)\,\d\Omega(\mu)<\infty.
	$$
	Then  any barycenter of $\Omega$  is absolutely continuous with respect to $\mm$.
\end{theorem}
\begin{proof}
If $\int_{\mathcal{P}(X)}{\rm Ent}_{\mm}(\mu)\,\d\Omega(\mu)=+\infty$, there is nothing to prove. Otherwise,  $\Omega$ is concentrated on $\De( {\rm Ent}_{\mm})$.
Since $\Omega$ has finite variance,  $\int_{\mathcal{P}(X)}W_2^2(\bar \mu, \mu)\,\d\Omega(\mu)<+\infty$, we can see that $\bar \mu$ has finite distance to  $\De( {\rm Ent}_{\mm})$. By assumption, $\bar \mu$ is the starting point of an ${\rm EVI}_K$ gradient flow of ${\rm Ent}_{\mm}$. Then \eqref{WJI1} follows from Theorem \ref{EVI JI}.
\end{proof}
\begin{remark}
Note that  the condition $\int_{\mathcal{P}(X)}{\rm Ent}_{\mm}(\mu)\,\d\Omega(\mu)<\infty$ implies  that $\Omega$ is concentrated on the absolutely continuous measures. If  $\ms$ is RCD, $\Omega$ is concentrated on finitely many absolutely continuous measures, in Theorem \ref{absolutely continuous} we will prove  the absolute continuity of the Wasserstein barycenter,  without the stronger condition $\int_{\mathcal{P}(X)}{\rm Ent}_{\mm}(\mu)\,\d\Omega(\mu)<\infty$.
\end{remark}

\begin{example}\label{coro:wn}Let  $(X,H,\gamma)$ be an abstract Wiener space equipped with the  Cameron--Martin distance  $\ds_H$ and the Gaussian measure $\gamma$.  Let  $\Omega$ be a  probability measure over $\mathcal{P}(X)$ with finite variance.  

Then  for any barycenter $\bar{\mu}$ of $\Omega$ it holds
	\begin{equation}\label{eq:wn}
	{\rm Ent}_{\gamma}(\bar{\mu})\leq \int_{\mathcal{P}(X)} {\rm Ent}_{\gamma}(\mu)\,\d\Omega(\mu)-\frac{1}{2}\int_{\mathcal{P}(X)} W_2^2(\bar{\mu},\mu)\,\d\Omega(\mu).
	\end{equation}
\end{example}
\begin{proof}
 Assume ${\mu}$  has finite distance from $\De( {\rm Ent}_{\gamma})$, then there is $\nu\in \De( {\rm Ent}_{\gamma})$ so that $W_2(\mu, \nu)<\infty$.  By \cite[Theorem 6.1]{FU}, there is a unique  optimal transport map  $T:={\rm Id}_X+\xi$ with $\xi\in H$ almost surely,  so that  $T_\sharp\nu=\mu$.  Let $(H_n)_{n \geq1}$ be an increasing sequence of regular subspaces of $H$ whose union is dense in $X$.  Let  $(\pi_n)_{n\geq 1}$  be the total increasing sequence of regular projections  (of $H$, converging to the identity map of $H$) associated to  $(H_n)_{n \geq1}$.  Denote also by $(\pi_n)_{n\geq 1}$
their continuous extensions to $X$ and define  $\pi^\perp_n={\rm Id}_X-\pi_n$.
 
For any $n\in N$, we define 
\[
F_n:=\big\{x\in X: \xi(x) \in H_n\big\}.
\]
and
$$\nu_n:=c_n \nu\llcorner_{F_n}\in \pr(X)$$ where $c_n$ is the normalizing constant.  We can see that $\mu_n:=T_\sharp (\nu_n) \to \mu$ in the Wasserstein space.
 
Denote $\mu_{n, t}:= ({\rm Id}_X+t\xi)_\sharp \nu_n$. We claim that  $\mu_{n, t}\ll \gamma$ for any $t\in (0,1)$.  To prove the claim, we adopt the same measure disintegration as in \cite[\S 6]{FU}. Let $\nu_n=\int_{H_n^\perp} \nu_n(\cdot | x^\perp)\,\d \nu_n^{\perp}(x^\perp)$ be a disintegration of $\nu_n$, where $\nu_n^{\perp}$ is  $(\pi^\perp_n)_\sharp \nu_n$  and $\nu_n(\cdot | x^\perp)$ is the regular conditional probability measure. Similarly, we disintegrate the Gaussian measure as $\gamma=\int_{H_n^\perp} \gamma(\cdot | x^\perp)\,\d \gamma^{\perp}(x^\perp)$. By \cite[Theorem 6.1]{FU} and its proof, we know $\gamma(\cdot | x^\perp), \nu_n(\cdot | x^\perp) $ are absolutely continuous with respect to  the $n$-dimensional Hausdorff measure  for almost every $x^\perp$. Since  $({\rm Id}_X+t\xi)_\sharp \nu_n$  is a Wasserstein geodesic, so $ ({\rm Id}_X+t\xi)_\sharp \nu_n(\cdot | x^\perp)$ is also a geodesic in the Wasserstein space $(\pr(H_n), W_2)$, so $({\rm Id}_X+t\xi)_\sharp \nu_n(\cdot | x^\perp)\ll \gamma(\cdot | x^\perp)$ for any $t\in [0,1)$. Then the claim follows from the following property
\[
\mu_{n, t}=({\rm Id}_X+t\xi)_\sharp \nu_n=\int_{H_n^\perp} ({\rm Id}_X+t\xi)_\sharp \nu_n(\cdot | x^\perp)\,\d \nu_n^{\perp}(x^\perp).
\]

  By \cite[\S11 and \S 13]{ambrosio2016optimal},  any measure in $\pr_{ac}(X, \ds_H, \gamma)$ is the starting point of an ${\rm EVI}_1$ gradient flow of ${\rm Ent}_{\gamma}$. Therefore for any $t\in (0,1)$, $\mu_{n,t}$ is the starting point of an ${\rm EVI}_1$ gradient flow.   Then  by completeness of the Wasserstein space (cf. \cite[Proposition 5.4]{FU}) and the Wasserstein contraction of ${\rm EVI}_1$ gradient flows,  we know  $\mu$ is the starting point of an ${\rm EVI}_1$ gradient flow. Then by Theorem \ref{WJI} we get \eqref{eq:wn}.
\end{proof}
\medskip

\begin{example}[See \cite{EH-Configure} and \cite{SS-configueII}]
Let  $(\bf \Upsilon, \ds_{\bf \Upsilon},\pi)$ be an extended metric measure space consisting of the configuration space  $\bf \Upsilon$ over a Riemannian manifold  $(M, g)$,   the  intrinsic metric  $\ds_{\bf \Upsilon}$ and the Poisson  measure $\pi$.   Assume that $(M, g)$ has Ricci curvature bounded below by $K \in \R$.

Then  for any probability measure  $\Omega$ over $\mathcal{P}(\bf \Upsilon)$ with finite variance, its barycenter $\bar{\mu}$ satisfies
	\begin{equation*}
	{\rm Ent}_{\pi}(\bar{\mu})\leq \int_{\mathcal{P}(X)} {\rm Ent}_{\pi}(\mu)\,\d\Omega(\mu)-\frac{K}{2}\int_{\mathcal{P}(X)} W_2^2(\bar{\mu},\mu)\,\d\Omega(\mu).
	\end{equation*}
\end{example}
\begin{proof}
This follows from Theorem \ref{WJI} and
\cite[Theorem 5.10.]{EH-Configure}.
\end{proof}
\medskip

\begin{example}\label{coro:rcd}Let  $\ms$ be an $\rcd$ space.  Let  $\Omega$ be a  probability measure over $\mathcal{P}_2(X, \ds)$ with finite variance.  
Then  for any barycenter $\bar{\mu}$ of $\Omega$,  it holds
	\begin{equation}
	{\rm Ent}_{\mm}(\bar{\mu})\leq \int_{\mathcal{P}_2(X, \ds)} {\rm Ent}_{\mm}(\mu)\,\d\Omega(\mu)-\frac{K}{2}\int_{\mathcal{P}_2(X, \ds)} W_2^2(\bar{\mu},\mu)\,\d\Omega(\mu).
	\end{equation}
\end{example}
\begin{proof}
By \cite[Theorem 5.1]{AGS-M}, any $\mu\in \pr_2(X, \ds)$ is the starting point of an ${\rm EVI}_K$ gradient flow. Then  the assertion follows from Theorem \ref{WJI}.
\end{proof}
\medskip

Applying the same argument  as \cite[Lemma 5.2]{AGS-M}, we can take advantage of the estimate \eqref{NearlyJI} to show the existence of Wasserstein barycenter, without local compactness of the underlying space. 
\begin{theorem}[Existence of the barycenter]\label{th:existenceRCD}
	Let $K\in \mathbb{R}$ and let $(X,\ds,\mm)$ be an extended metric measure space,  $\Omega$ be a  probability measure over $\mathcal{P}(X)$ with finite variance and $$
	\int_{\mathcal{P}(X)}{\rm Ent}_{\mm}(\mu)\,\d\Omega(\mu)<\infty.
	$$  
	
	Assume that one of the following conditions holds:
\begin{itemize}
\item [{\bf A.}]  $(X, \ds, \mm)$ is an $\rcd$ metric measure space and $\Omega$ is concentrated on $\mathcal{P}_2(X, \ds)$;
\item [{\bf B.}]  $\mm$ is a probability measure,  any  ${\mu} \in \pr(X)$ which has finite distance to $\De( {\rm Ent}_{\mm})$ is the starting point of an ${\rm EVI}_K$ gradient flow of ${\rm Ent}_{\mm}$ in the Wasserstein space.
\end{itemize} 
Then $\Omega$ has a barycenter.
\end{theorem}
\begin{proof}
For any $0<\epsilon<1$,   by  assumption that $\Omega$ has  finite  variance, there exists  $\bar{\mu}^{\epsilon}\in \pr(X)$ such that
	\begin{equation}\label{eq:var}
	\int_{\mathcal{P}(X)} W_2^2(\mu,\bar{\mu}^{\epsilon})\,\d\Omega(\mu)\leq {\rm Var}(\Omega)+\epsilon.
	\end{equation}
	By  assumption and Theorem \ref{EVI JI}, inequality \eqref{NearlyJI},  we have
	\begin{equation}\label{NearlyJI1}
	{\rm Ent}_\mm(\bar{\mu}_t^{\epsilon})\leq \int_{{\mathcal{P}(X)}} {\rm Ent}_\mm(\mu)\,\d\Omega(\mu)-\frac{K}{2}{\rm Var}(\Omega)+\frac{\epsilon}{2I_K (t)},
	\end{equation}
	where $\bar{\mu}_t^{\epsilon}$ denotes the ${\rm EVI}_K$ gradient flow of ${\rm Ent}_\mm$  starting from $\bar{\mu}^{\epsilon}$.  Set $\bar{\nu}_{\epsilon}:=\bar{\mu}_{\epsilon}^{\epsilon}$.  It can be seen  that $\epsilon/I_K(\epsilon)$ is  bounded for $0<\epsilon<1$. Combining this with \eqref{NearlyJI1}, we know that the set of measures $\left\{\bar{\nu}_{\epsilon}\right\}_{\epsilon \in (0,1)}$ has uniformly bounded entropy. Thus, it is tight (cf. \cite[Lemma 4.4]{AGMR-R} and \cite[Theorem 1.2]{FSS10}). Without loss of generality, we   assume that $\bar{\nu}_{\epsilon}$ converges to  $\bar{\nu}\in \mathcal{P}(X)$ in the weak sense as $\epsilon\rightarrow 0$.
	
	We claim that $\bar{\nu}$ is a barycenter of $\Omega$. 	For any   ${\rm EVI}_K$ gradient flows of ${\rm Ent}_\mm$  starting from $x_0, y_0$,  it is  known that (cf.  \cite[Chapter 4]{AGS-G}, \cite[Proposition 2.22]{AGS-M} and \cite[\S 10]{ambrosio2016optimal})
	\begin{equation}\label{contration of gradient flow}
		W_2(x_t,y_t)\leq e^{-Kt}W_2(x_0, y_0)~~~~\forall t>0.
	\end{equation}
Thus
	\begin{equation}\label{contration1}
		W_2(\mu,\bar{\nu}_{\epsilon})-W_2(\mu_\epsilon,\mu)\leq W_2(\mu_\epsilon,\bar{\nu}_{\epsilon})\leq e^{-K\epsilon}W_2(\mu, \bar{\mu}^{\epsilon}),
	\end{equation}
	where  $\mu_\epsilon$ is the ${\rm EVI}_K$ gradient flow of ${\rm Ent}_\mm$   from $\mu$. 
	
	Integrating \eqref{contration1} in $\mu$ with respect to $\Omega$, we get
	\begin{equation}\label{123}
		\int_{\mathcal{P}(X)}W_2^2(\mu,\bar{\nu}_{\epsilon})\,\d\Omega(\mu)\leq \int_{\mathcal{P}(X)}\big(e^{-K\epsilon}W_2(\mu, \bar{\mu}^{\epsilon})+W_2(\mu_\epsilon, \mu)\big)^2\,\d\Omega(\mu).
	\end{equation}
By \eqref{eq:integral version}, we have
$$
\frac{e^{K\epsilon}}{2}W_2^2(\mu,\mu_{\epsilon})\leq I_K(\epsilon){\rm Ent}_{\mm}(\mu).
$$
Then by dominated convergence theorem, we get 
\begin{equation}\label{1231}
\lim_{\epsilon\rightarrow 0} \int_{{\mathcal{P}(X)}} W_2^2(\mu,\mu_{\epsilon})\,\d\Omega(\mu)=\int_{{\mathcal{P}(X)}}  \lim_{\epsilon\rightarrow 0}W_2^2(\mu,\mu_{\epsilon})\,\d\Omega(\mu)=0.
\end{equation}
Combining with the lower semi-continuity
	 $$
	 W_2(\bar{\nu},\mu)\leq\lmti{\epsilon}{ 0}W_2(\bar{\nu}_\epsilon,\mu),\qquad \text{for any } \mu\in\mathcal{P}_2(X),
	 $$
	and \eqref{123},  \eqref{1231},  we   get 
	$$
	\begin{aligned}
	{\rm Var}(\Omega)&\leq \int_{\mathcal{P}(X)}W_2^2(\mu,\bar{\nu})\,\d\Omega(\mu)\\ &\leq \int_{{\mathcal{P}(X)}}\lmti{\epsilon}{ 0}W_2^2(\mu,\bar{\nu}_{\epsilon})\,\d\Omega(\mu)\\&\leq \lmti{\epsilon}{0}\int_{\mathcal{P}(X)}W^2_2(\mu, \bar{\nu}_{\epsilon})\,\d\Omega(\mu)\\&\leq \lmti{\epsilon}{0}\int_{\mathcal{P}(X)}\big(e^{-K\epsilon}W_2(\mu, \bar{\mu}^{\epsilon})+W_2(\mu_\epsilon, \mu)\big)^2\,\d\Omega(\mu)\\&\leq \lmti{\epsilon}{0}\int_{\mathcal{P}(X)} \left [e^{-2K\epsilon}(1+\delta)W_2^2(\mu, \bar{\mu}^{\epsilon})+(1+{\delta}^{-1})W_2^2(\mu_\epsilon, \mu) \right]\,\d\Omega(\mu)\\&\leq \lmti{\epsilon}{0}\int_{\mathcal{P}(X)}(1+\delta)W^2_2(\mu, \bar{\mu}^{\epsilon})\,\d\Omega(\mu)
	\end{aligned}
	$$
	for any $\delta>0$.
	Letting $\delta \to 0$ and combining  with \eqref{eq:var}, we prove the claim.
\end{proof}
\medskip

Similar to Proposition \ref{uniqueness in finite space}, we can prove the uniqueness of Wasserstein barycenter, under a slightly weaker Monge property thanks to Theorem  \ref{th:existenceRCD}. Important examples fitting our setting includes $\rcd$ spaces \cite{gigli2016optimal} and abstract Wiener spaces \cite{FSS10,FU}.

\begin{theorem}[Uniqueness of the barycenter]\label{uniqueness in infinite space}

	Let  $(X,\ds,\mm)$ be an extended metric measure  space satisfying one of the two conditions in Theorem \ref{th:existenceRCD} and satisfying the following weak Monge property: for any $\mu,\nu\in  \pr_{ac}\ms$ with $W_2(\mu, \nu)<\infty$, there exists a unique optimal transport map between $\mu$ and $\nu$.  
	
Then for any $\Omega\in \pr_2(\pr(X), W_2)$ with $$\int_{\pr(X)}{\rm Ent}_\mm(\mu)\,\d\Omega(\mu)<\infty,$$  there exists a unique Wasserstein barycenter of $\Omega$.
\end{theorem}

\begin{proof}
	By Theorem \ref{th:existenceRCD}, we know $\Omega$ has a Wasserstein barycenter $\bar{\mu}$.  By  Jensen's inequality in Theorem \ref{WJI},
	\begin{equation}
		{\rm Ent}_\mm(\bar{\mu})\leq \int_{\pr(X)}{\rm Ent}_\mm(\mu)\,\d\Omega(\mu)-\frac{K}{2}\int_{\pr(X)} W_2^2(\bar{\mu},\mu)\,\d \Omega(\mu)<\infty.
	\end{equation}
	This implies that the Wasserstein barycenter $\bar{\mu}$ is absolutely continuous with respect to $\mm$. By assumption, for any $\mu,\nu\in \pr(X)$ with $W_2(\mu, \nu)<\infty$ and $\mu,\nu \ll m$, there exists a unique optimal transport map from $\mu$ to $\nu$. 
	
	Assume by contradiction,  there exist two different Wasserstein barycenters $\bar{\mu}_1,\bar{\mu}_2$ of $\Omega$, so that
	\begin{equation}\label{5.21}
	\int_{\pr(X)}W_2^2(\mu, \bar{\mu}_i)\,\d\Omega(\mu)=\min_{\nu\in\pr(X)}\int_{\pr(X)}W_2^2(\mu, \nu)\,\d\Omega(\mu)~~~~i=1,2.
	\end{equation}
Using the same argument as in the proof of Proposition \ref{uniqueness in finite space}, we can prove
	\begin{equation}\label{5.22}
		W_2^2(\mu, (\bar{\mu}_1+\bar{\mu}_2)/2)<\frac 12 \Big( W_2^2(\mu, \bar{\mu}_1)+W_2^2(\mu, \bar{\mu}_2) \Big).
	\end{equation}
Denote  $\bar{\mu}=(\bar{\mu}_1+\bar{\mu}_2)/2\in \pr(X)$.  Integrating \eqref{5.22} with respect to $\Omega$ and noticing that $\int_{\pr(X)}{\rm Ent}_\mm(\mu)\,\d\Omega(\mu)<\infty$ implies $\Omega(\pr_{ac}\ms)=1$, we obtain 
	\begin{equation}
		\begin{aligned}
			\int_{\pr(X)}W_2^2(\mu, \bar{\mu})\,\d\Omega(\mu)
			&<\frac 12 \int_{\pr(X)}W_2^2(\mu, \bar{\mu}_1)\,\d\Omega(\mu)+\frac 12 \int_{\pr(X)}W_2^2(\mu, \bar{\mu}_2)\,\d\Omega(\mu)\\
			&=\min_{\nu\in\pr(X)}\int_{\pr(X)}W_2^2(\mu, \nu)\,\d\Omega(\mu),
		\end{aligned}
	\end{equation}
	which is a contradiction. Thus the Wasserstein barycenter of $\Omega$ is unique.
\end{proof}

\subsection{${\rm EVI}_{K,N}$ gradient flows}
 Similar to Theorem \ref{EVI JI}, using a
finite dimensional formulation of the gradient flow,  ${\rm EVI}_{K,N}$ gradient flow,  we can prove a finite dimensional version of 
Jensen's inequality, which is new even on $\mathbb{R}^n$. For simplicity, throughout  this  subsection,  we assume that $(X, \ds)$ is a metric space.

 Firstly, we prove the following basic estimates.
\begin{lemma}\label{Estimat EVI -K,N}
	Let $(X,\ds)$ be a metric space,  $K\neq 0$,$N\in (0, \infty)$ and  $U_N$,  $(y_t)$ be as in Definition \ref{def:EVI K,N}. Then  for  $L= \pi \sqrt{\frac N{|K|}}\vee 2$,  there exists $C_1=C_1(K, N)>0$, such that for $t$ sufficiently small, it holds the following uniform estimate (in $z$):
	\begin{itemize}
		\item [(i)] for  $z\in X$ with $\ds(y,z)\leq L$, $\ds(y_t,z)\leq C_1$.
	
	\item [(ii)] for $z\in X$ with $\ds(y,z)> L$, $\ds(y_t,z)\leq  L\ds(y,z)$.
	\end{itemize}

\end{lemma}
\begin{proof}
For $K>0$, by Definition \ref{def:EVI K,N} we can see that $(X, \ds)$ is bounded,   there is nothing to prove.  So we  assume $K< 0$.

Recall that the ${\rm EVI}_{K,N}$ gradient flow  satisfies
	\begin{equation}\label{eVI -k,n}
	\frac{\d }{\dt}s_{K/N}^2 \left(\frac{1}{2}\ds(y_t,z)\right)+Ks_{K/N}^2\left(\frac{1}{2}\ds(y_t,z)\right)\leq	\frac{N}{2}\left(1-\frac{U_N(z)}{U_N(y_t)}\right)
	\end{equation}
 for a.e. $t>0$.  Its integral version   (see \cite[Proposition 2.18]{EKS-O} and  Proposition \ref{integral version}) is 
	\begin{equation}\label{Integral eVI -K,N}
	e^{Kt}s_{K/N}^2\left(\frac{1}{2}\ds(y_t,z)\right)-s_{K/N}^2\left(\frac{1}{2}\ds(y,z)\right)\leq I_K(t)\frac{N}{2}\left(1-\frac{U_N(z)}{U_N(y_t)}\right).
	\end{equation}
Notice that $U_N\geq 0$,  by \eqref{Integral eVI -K,N}, we get
	\begin{equation}\label{estimate 1}
	c_{K/N}(\ds(y_t,z))\leq e^{-Kt}c_{K/N}(\ds(y,z))+\sqrt{-\frac{N}{K}}\big(1-e^{-Kt}-K e^{-Kt}I_K(t)\big),
	\end{equation}
	where we  use the identity  $-\frac KN s_{K/N}^2 (\frac{x}{2})=\frac{\sqrt{-K/N}c_{K/N}(x)-1}{2}$. Then the first statement of the lemma follows. 
	
	\medskip
	
	Next,  for $z\in X$ with $\ds(y,z)> L$,   we prove the  assertion by contradiction.  Assume that $\ds(y_t,z)\geq L\ds(y,z)$. By monotonicity of $x\mapsto \sinh x$ and $\mapsto\cosh x$ on $\left[0,\infty\right)$, we have
	\begin{equation}\label{estimate 2}
	\begin{aligned}
	c_{K/N}(\ds(y_t,z))&\geq c_{K/N}(2\ds(y,z))= \sqrt{-\frac{K}{N}}\big(c_{K/N}^2(\ds(y,z))+ s_{K/N}^2(\ds(y,z))\big)\\
	&\geq \cosh(\pi)c_{K/N}(\ds(y,z))+\sqrt{-\frac{N}{K}}\sinh^2(\pi),
	\end{aligned}
	\end{equation}
	When $t$ is sufficiently small, such that 
	$$e^{-Kt}<\cosh(\pi)$$ 
	and 
	$$1-e^{-Kt}-Ke^{-Kt}I_K(t)<\sqrt{-\frac{N}{K}}\sinh^2(\pi),$$
we can see that \eqref{estimate 2} contradicts to \eqref{estimate 1}. 
\end{proof}

\medskip

Now we  can prove a Jensen-type inequality with dimension parameter:
	
\begin{theorem}[Jensen-type inequality with dimension parameter]\label{JIEVI -K,N}
	Let $(X,\ds)$ be a metric space,  $N, K$, $U_N$ and $(y_t)$ be as in Lemma \ref{Estimat EVI -K,N}. Let $\mu$ be a  probability measure on $X$ with finite variance. Assume that ${\rm EVI}_{K,N}$ gradient flow of $E$ exists for any initial data $y_0\in X$.  Then for any   barycenter  $\bar{y}$ of $\mu$,  the following Jensen-type inequality holds:
	\begin{equation}\label{5.11}
\frac{1}{U_N(\bar{y})}	\int_{X} \frac{\ds(\bar{y},z)}{s_{K/N}(\ds(\bar{y},z))}U_N(z)\,\d\mu(z)\leq\int_{X}\frac{\ds(\bar{y},z)}{t_{K/N}(\ds(\bar{y},z))}\,\d\mu(z)<+\infty.
	\end{equation}
	where $t_{K/N}:=s_{K/N}/c_{K/N}$.
\end{theorem}
\begin{proof}
	 Recall that the  ${\rm EVI}_{K,N}$ gradient flow  $(y_t)$ from $y_0$ satisfies
	\begin{equation}\label{EVI -k,n}
	\frac{\d }{\dt}s_{K/N}^2\left(\frac{1}{2}\ds(y_t,z)\right)+Ks_{K/N}^2\left(\frac{1}{2}\ds(y_t,z)\right)\leq	\frac{N}{2}\left(1-\frac{U_N(z)}{U_N(y_t)}\right)
	\end{equation}
 for a.e. $t>0$. 
 It is known that \eqref{EVI -k,n} is equivalent to the following inequality (cf. \cite[Lemma 2.15, (2.20)]{EKS-O}):
	\begin{equation}\label{EVI -K,N 1}
	\frac{1}{2}\frac{\d}{\dt}\ds^2(y_t,z)\leq\frac{N\ds(y_t,z)}{s_{K/N}(\ds(y_t,z))}\left[c_{K/N}(\ds(y_t,z))-\frac{U_N(z)}{U_N(y_t)}\right].
	\end{equation}
	And its integral version is
	\begin{equation}\label{Integral EVI -K,N 1}
	\frac{1}{2}\ds^2(y_t,z)-\frac{1}{2}\ds^2(y,z)\leq N\left[\int_{0}^{t}\frac{\ds(y_s,z)}{t_{K/N}(\ds(y_s,z))}-\frac{{ U_N(z)}\ds(y_s,z)}{{U_N(y_s)}s_{K/N}(\ds(y_s,z))}\,\d s\right].
	\end{equation}

	Integrating \eqref{Integral EVI -K,N 1} in $z$ with the probability measure $\mu$,   choosing $y_0=\bar{y}$ as the  barycenter  and $(y_t)$ as the ${\rm EVI}_{K,N}$ gradient flow of $U_N$ starting from $\bar{y}$,   we get
	\begin{equation}\label{EVI -K,N JI1}
	\begin{aligned}
	&\frac{1}{2}\int_{X}\ds^2(y_t,z)\,\d\mu(z)-\frac{1}{2}\int_{X}\ds^2(\bar{y},z) \,\d\mu(z)\\ \leq &N\left[\int_{X}\int_{0}^{t}\frac{\ds(y_s,z)}{t_{K/N}(\ds(y_s,z))}\,\d s\,\d\mu(z)-\int_{X}\int_{0}^{t}\frac{ U_N(z)}{U_N(y_s)}\frac{\ds(y_s,z)}{s_{K/N}(\ds(y_s,z))}\,\d s \,\d\mu(z)\right].
	\end{aligned}		
	\end{equation}  
	By Definition \ref{def:barycenter}, we know that $\int_{X}\ds(y_t,z)^2 \,\d\mu(z)\geq\int_{X}\ds(\bar{y},z)^2\,\d\mu(z)$. Combining this with \eqref{EVI -K,N JI1}, we get
	\begin{equation}\label{EVI -K,N JI2}
	0\leq N\left[\int_{X}\int_{0}^{t}\frac{\ds(y_s,z)}{t_{K/N}(\d(y_s,z))}\,\d s\,\d\mu(z)-\int_{X}\int_{0}^{t}\frac{ U_N(z)}{U_N(y_s)}\frac{\ds(y_s,z)}{s_{K/N}(\ds(y_s,z))}\,\d s\,\d\mu(z)\right]. 
	\end{equation}
	
Rearrange \eqref{EVI -K,N JI2} and divide both sides of it by $Nt$:
	\begin{equation}\label{EVI -K,N JI3}
	\int_{X} \left[\frac{1}{t}\int_{0}^{t}\frac{U_N(z)}{U_N(y_s)}\frac{\ds(y_s,z)}{s_{K/N}(\ds(y_s,z))}\,\d s\right]\,\d\mu(z)\leq\int_{X}\left[\frac{1}{t}\int_{0}^{t}\frac{\ds(y_s,z)}{t_{K/N}(\ds(y_s,z))}\,\d s\right]\,\d\mu(z).
	\end{equation}
	
	Note that $\lim_{s\rightarrow 0}\d(y_s,z)=\d(\bar{y},z)$ and $x/s_{K/N}(x)$ is continuous and bounded. Combining with the upper semi-continuity  of the non-negative functional $U_N$, by Fatou's lemma, we obtain
	$$
	\liminf_{t\rightarrow 0}\frac{1}{t}\int_{0}^{t}\frac{\ds(y_s,z)}{{U_N(y_s)}s_{K/N}(\ds(y_s,z))}\,\d s\geq \frac{\ds(\bar{y},z)}{{U_N(\bar y)}s_{K/N}(\ds(\bar{y},z))}.
	$$
	 By Fatou's lemma again, we deduce
	\begin{equation}\label{fatou lemma}
	\begin{aligned}
	&\frac{1}{U_N(\bar{y})}\int_{X} \frac{\ds(\bar{y},z)}{s_{K/N}(\ds(\bar{y},z))}U_N(z)\,\d\mu(z)\\
	\leq&\liminf_{t\rightarrow 0}\int_{X} \left[\frac{1}{t}\int_{0}^{t}\frac{U_N(z)}{U_N(y_s)}\frac{\ds(y_s,z)}{s_{K/N}(\ds(y_s,z))}\,\d s\right]\,\d\mu(z).
	\end{aligned}
	\end{equation} 
	
	\medskip
	
	Next, we consider  separately the following two cases: $K>0$ and $K<0$. (For $K=0$, the proof is similar to Theorem \ref{EVI JI}.)
	
	\paragraph{Case 1:} $K>0$. From the definition, we can see that the diameter of $(X,\d)$ is bounded  from above by $\pi \sqrt{\frac N{K}}$. Thus, the function $(y,z) \mapsto \frac{\ds(y,z)}{t_{K/N}(\ds(y,z))}$ is bounded. Notice that $\mu$ is a probability measure. By  Fatou's lemma, we get
	\begin{equation}\label{fatou lemma 2}
	\begin{aligned}
	&\limsup_{t\rightarrow 0}\int_{X} \left[\frac{1}{t}\int_{0}^{t}\frac{\ds(y_s,z)}{t_{K/N}(\ds(y_s,z))}\,\d s \right]\,\d\mu(z)\\ \leq&\int_{X}\limsup_{t\rightarrow 0}\left[\frac{1}{t}\int_{0}^{t}\frac{\ds(y_s,z)}{t_{K/N}(\ds(y_s,z))}\,\d s \right]\,\d\mu(z)\\
	=&\int_{X}\frac{\ds(\bar{y},z)}{t_{K/N}(\ds(\bar{y},z))}\,\d\mu(z)< +\infty.	
	\end{aligned}
	\end{equation} 
	Combining \eqref{EVI -K,N JI3}, \eqref{fatou lemma} and \eqref{fatou lemma 2}, we get \eqref{5.11}.
	
\paragraph{Case 2:} $K<0$. To estimate the right side of \eqref{EVI -K,N JI3} with  Lemma \ref{Estimat EVI -K,N},   we write:
	\begin{eqnarray*}
	&&\int_{X}\frac{1}{t}\int_{0}^{t}\frac{\ds(y_s,z)}{t_{K/N}(\ds(y_s,z))}\,\d s\,\d\mu(z)\\
	&=& \underbrace{\int_{\{z\in X:  \ds(\bar{y},z)\leq L\}}\frac{1}{t}\int_{0}^{t}\frac{\ds(y_s,z)}{t_{K/N}(\ds(y_s,z))}\,\d s\,\d\mu(z)}_{A(t)}\\
	&&+ \underbrace{\int_{\{z\in X:  \ds(\bar{y},z)> L\}}\frac{1}{t}\int_{0}^{t}\frac{\ds(y_s,z)}{t_{K/N}(\ds(y_s,z))}\,\d s\,\d\mu(z)}_{B(t)}.
	\end{eqnarray*}
	
By  Lemma \ref{Estimat EVI -K,N}-(i) and the monotonicity of $x\mapsto \frac{x}{\tanh x}$, we have
\[
\frac{\ds(y_s,z)}{t_{K/N}(\ds(y_s,z))}\leq \frac{C_1}{t_{K/N}(C_1)}~~\text{for all}~z\in X ~\text{with}~ \ds(\bar{y},z)\leq  L.
\]
So by dominated convergence theorem,  we have
\begin{equation}\label{eq1:bckn}
\lmt{t}{0} A(t)=\int_{\{z\in X: \ds(\bar{y},z)\leq L\}}\frac{\ds( \bar{y},z)}{t_{K/N}(\ds(\bar{y},z))}\,\d\mu(z).
\end{equation}
	
	By  Lemma \ref{Estimat EVI -K,N}-(ii) and the inequality $t_{K/N}(x)\geq \frac{ t_{K/N}(L)L}{x}$ for $x>L$, we have
	\[
	\frac{1}{t}\int_{0}^{t}\frac{\ds(y_s,z)}{t_{K/N}(\ds(y_s,z))}\leq \frac 1{t_{K/N}(L)L}\ds(\bar{y},z)^2 ~~\text{for all}~z\in X ~\text{with}~ \ds(\bar{y},z)> L.
	\]
Note that  the function $z\mapsto \ds(\bar{y},z)^2$ is $\mu$-integrable, so by dominated convergence theorem we get
\begin{equation}\label{eq2:bckn}
\lmt{t}{0}B(t)= \int_{\{z\in X: \ds(\bar{y},z)> L\}}\frac{\ds(\bar{y},z)}{t_{K/N}(\ds(\bar{y},z))}\,\d\mu(z).
\end{equation}

 Combining  \eqref{EVI -K,N JI3}, \eqref{fatou lemma}, \eqref{eq1:bckn} and \eqref{eq2:bckn}, we get 
	\begin{equation*}
\frac{1}{U_N(\bar{y})}	\int_{X} \frac{\ds(\bar{y},z)}{s_{K/N}(\ds(\bar{y},z))}U_N(z)\,\d\mu(z)\leq\int_{X}\frac{\ds(\bar{y},z)}{t_{K/N}(\ds(\bar{y},z))}\,\d\mu(z)\leq \frac 1{t_{K/N}(L)L}{\rm Var}(\mu)+\frac{C_1}{\tanh(C_1)}
	\end{equation*}
	which is the thesis.
\end{proof}
\medskip

 Combining Theorem \ref{th:rcdknevi}, Theorem \ref{th:existenceRCD} and  Theorem  \ref{JIEVI -K,N},   we obtain the following corollaries.
	
\begin{corollary}\label{Jensen's equality finite}
Let $N\in [1,\infty)$. Let $(X,\ds,\mm)$ be an ${\rm RCD}(K,N)$ space and $\Omega$ be a Borel probability measure over $\mathcal{P}_2(X, \ds)$ with finite variance. Then it holds the following  Jensen-type inequality:
\begin{equation}\label{JIKN}
\int \frac{W_2(\bar{\mu},\mu)}{s_{K/N}(W_2(\bar{\mu},\mu))}U_N(\mu)\,\d\Omega(\mu)\leq  U_N(\bar{\mu}) \int \frac{W_2(\bar{\mu},\mu)}{t_{K/N}(W_2(\bar{\mu},\mu))}\,\d\Omega(\mu),
\end{equation}
where $\bar{\mu}$ is a barycenter of $\Omega$,    $U_N(\mu)=e^{-\frac{{\rm Ent}_{\mm}(\mu)}{N}}$.
\end{corollary}

\begin{corollary}\label{entropy finite EVI -k,N}
Let $N\in [1,\infty) $ and $(X,\ds,\mm)$ be an ${\rm RCD}(K,N)$ space.  Let $\Omega$ be a Borel probability measure over $\mathcal{P}_2(X, \ds)$ with finite variance, which gives positive mass to the set $\De( {\rm Ent}_{\mm})=\left\{\mu: {\rm Ent}_{\mm}(\mu)<\infty\right\}$. Then the entropy of the barycenter of $\Omega$ is finite. In particular, the barycenter  is absolutely continuous with respect to $\mm$ and is unique.
\end{corollary}
\begin{proof}
Let $\bar{\mu}$ be a barycenter of $\Omega$.  Since $\Omega(\De( {\rm Ent}_{\mm}))>0$,  we have
	$$
	\int_{\mathcal{P}_2(X, \ds)} \frac{W_2(\bar{\mu},\mu)}{s_{K/N}(W_2(\bar{\mu},\mu))}U_N(\mu)\,\d\Omega(\mu)>0.
	$$
From \eqref{JIKN}, we can see that $U_N(\bar{\mu})=e^{-\frac{{\rm Ent}_{\mm}(\bar{\mu})}{N}}>0$ and ${\rm Ent}(\bar{\mu})<\infty$.  By Theorem \ref{uniqueness in infinite space} we know the barycenter is unique.
\end{proof}

\subsection{Multi-marginal optimal transport}
In this subsection we will study multi-marginal optimal transport problem  in the setting of ${\rm RCD}$  metric measure spaces. This is a continuation of Gangbo and Swiech's  study in Euclidean setting \cite{gangbo1998optimal}, Kim and Pass's   in Riemannian manifolds setting \cite{kim2015multi} and Jiang's   \cite{JiangWasserstein} in Alexandrov space setting.  It is  worth to emphasize that,  the relationship between multi-marginal optimal transport problem and Wasserstein barycenter problem,  was observed and investigated by Carlier--Ekeland \cite{carlier2010matching},  and  by Agueh--Carlier in their seminal paper \cite{AguehCarlier}. In the following discussion,   we focus on the  cost function $$c(x_1,\dots,x_n)=\inf_{y\in X}\sum_{i=1}^n  \ds^2(x_i,y)\in [0, +\infty].$$
Note that for $n=2$, the multi-marginal optimal transport problem is equivalent to finding the mid-point of two given measures.

 Before stating our main theorem, we summarize   some results which have been proved in the previous sections.

\begin{proposition}\label{finite support}
	Let $(X,\ds,\mm)$ be a metric measure space,  $\mu_1,\dots,\mu_n\in \pr_2(X, \ds)$.  Assume one of the following conditions holds:
	\begin{itemize}
	\item  [{\bf A.}] $(X,\d,\mm)$ is an ${\rm RCD}(K,N)$ space, and  ${\rm Ent}_\mm(\mu_1)<\infty$;
	\item [{\bf B.}] 	 $(X,\d,\mm)$ is an  ${\rm RCD}(K,\infty)$ space, and  ${\rm Ent}_\mm(\mu_i)<\infty$ for any $i=1,\dots,n$.
	\end{itemize}
Then there exists a unique Wasserstein barycenter  of the measure $\frac 1n \sum_{i=1}^n \delta_{\mu_i}$,  and the barycenter is absolutely continuous with respect to $\mm$.

\end{proposition}

\begin{proof}
By Proposition \ref{proper space} and Theorem \ref{th:existenceRCD}, we have the existence of  Wasserstein barycenter for $\rcdkn$ and $\rcd$ spaces respectively. From Proposition \ref{uniqueness in finite space} and Theorem \ref{uniqueness in infinite space}, we know that Wasserstein barycenter is unique.
	
The absolute continuity of the  Wasserstein barycenter follows from Corollary \ref{entropy finite EVI -k,N} and Theorem \ref{WJI} respectively.
\end{proof}

Using Proposition \ref{finite support}, we can prove the existence and the uniqueness of  the multi-marginal optimal transport map,  in case the marginal measures have finite entropy.

\begin{proposition}\label{finite entropy MOT}
	Let $(X,\ds,\mm)$ be a metric measure space,  $\mu_1,\dots,\mu_n\in \pr_2(X, \ds)$. Then the multi-marginal optimal transport problem of Monge type, associated with the cost function $c(x_1,\dots,x_n)$,   has a unique solution,  if one of the following conditions holds:
	\begin{itemize}
	\item  [{\bf A.}] $(X,\ds,\mm)$ is an ${\rm RCD}(K,N)$ space, and  ${\rm Ent}_\mm(\mu_1)<\infty$;
	\item [{\bf B.}] 	 $(X,\ds,\mm)$ is an ${\rm RCD}(K,\infty)$ space, and  ${\rm Ent}_\mm(\mu_i)<\infty$ for any $i=1,\dots,n$.
	\end{itemize}
	In particular, any multi-marginal  optimal transport plan $\pi$ is concentrated on the graph of a $X^{n-1}$-valued map, and for $\pi$-a.e. ${\bf x}:=(x_1, \dots, x_n)\in X^n$,  the measure $\frac 1n \sum_{i=1}^n \delta_{x_i}$  has a unique barycenter in $X$.
 
\end{proposition}
\begin{proof}
	First of all,  by \eqref{eq:mmg1} in the proof of  Theorem \ref{th:sp}, we have
	\begin{equation}\label{536}
		\inf_{\nu\in \pr(X)}\sum_{i=1}^n W_2^2(\mu_i,\nu)\leq \inf_{\pi\in\Pi}\int_{X^n}c(x_1,\dots,x_n)\,\d\pi.
	\end{equation}

By Proposition  \ref{finite support}, we know  $\frac 1n \sum \delta_{\mu_i}$ has a unique Wasserstein barycenter $\bar{\mu}$,   which is absolutely continuous with respect to $\mm$.  For  $i\in \{1,\dots,n\}$,  by \cite{cavalletti2017optimal} and \cite{gigli2016optimal} respectively, there exists a unique optimal transport map $T_i$ from $\bar{\mu}$ to $\mu_i$.  In addition,  there is a unique optimal transport map from $\mu_1$ to $\bar{\mu}$, which is the inverse  of $T_1$ and we denote it by $T_1^{-1}$. In particular,  we have
	$$W_2^2(\mu_i, \bar{\mu})=\int_{X} \ds^2(x,T_i({\bf x}))\,\d\bar{\mu}, \quad i=1,\dots,n,$$
	and
	$$W_2^2(\mu_1, \bar{\mu})=\int_{X} \ds^2(x_1,T_1^{-1}(x_1))\,\d{\mu_1}.$$
	
	We claim $(T_2\circ T_1^{-1},\dots,T_n\circ T_1^{-1})$ is a multi-marginal optimal transport map. It is surely a multi-marginal transport map, and
	\begin{equation}\label{8.4}
		\begin{aligned}
			\inf_{\pi\in\Pi}\int_{X^n}c(x_1,\dots,x_n)\,\d\pi
			&\leq \int_{X}c\left(x_1,T_2\circ T_1^{-1}(x_1),\dots,T_n\circ T_1^{-1}(x_1)\right)\,\d\mu_1\\
			&= \int_{X}\inf_{y\in X}\sum_{i=1}^n \ds^2(T_i\circ T_1^{-1}(x_1),y)\,\d\mu_1\\
			&\leq \int_{X}\sum_{i=1}^n \ds^2(T_i\circ T_1^{-1}(x_1),T_1^{-1}(x_1))\,\d\mu_1\\
			&=\int_{X}\sum_{i=1}^n \ds^2(T_i(z),z)\,\d ({T_1^{-1}})_{\sharp}\mu_1(z)\\
			&=\int_{X}\sum_{i=1}^n \ds^2(T_i(z),z)\,\d \bar{\mu}(z)\\
			&\overset{*}=\min_{\nu\in\pr(X)}\sum_{i=1}^n W_2^2(\mu_i,\nu)\overset{\eqref{536}}\leq \inf_{\pi\in \Pi}\int_{X^n}c(x_1,\dots,x_n)
			\,\d\pi,
		\end{aligned}
	\end{equation}
	where $(*)$ holds since $\bar \mu$ is the Wasserstein barycenter and $$\int_{X} \d^2(T_i(z),z)\,\d \bar{\mu}(z)= W_2^2(\mu_i,\bar{\mu}),~~~i=1,2...,n.$$
	Therefore, the inequalities  above are all equalities. In particular,
	\begin{equation}
		\inf_{\pi\in\Pi}\int_{X^n}c(x_1,\dots,x_n)\,\d\pi=\int_{X}c(x_1,T_2\circ T_1^{-1}(x_1),\dots,T_n\circ T_1^{-1}(x_1))\,\d\mu_1.
	\end{equation}
	This means,  any optimal transport plan $\pi$ is induced by $(T_2\circ T_1^{-1},\dots,T_n\circ T_1^{-1})$  is a multi-marginal optimal transport map.
Furthermore, from the second inequality in \eqref{8.4}, we can see that $T^{-1}_1(x_1)$ is a barycenter of $\frac 1n \sum_i \delta_{T_i \circ T_1^{-1}(x_1)}$.
\medskip

	Next, we  will prove the uniqueness of  the multi-marginal optimal transport map.  By \eqref{8.4}, we have
	\begin{equation}
		\inf_{\pi\in\Pi}\int_{X^n}c(x_1,\dots,x_n)\,\d\pi=\int_{X}\sum_{i=1}^n \ds^2(T_i(z),z)\,\d \bar{\mu}(z).
	\end{equation}
Assume,  for sake of contradiction, there exists    is  a multi-marginal optimal transport map $(F_2,\dots,F_n)$, which is distinct with $(T_2\circ T_1^{-1},\dots,T_n\circ T_1^{-1})$.

	Denote $F_1(x_1)=x_1$.  By Lemma  \ref{measurable selection} and the first part of the proof, there exists a measurable  map $S: X \to X$ satisfying
	\begin{equation}
		c(x_1,F_2(x_1),\dots,F_n(x_1)) :=\inf_{y\in X}\sum_{i=1}^n \ds^2(F_i(x_1),y)=\sum_{i=1}^n \ds^2(F_i(x_1),S(x_1)).
	\end{equation}

	Then we have 
	\begin{equation}
		\begin{aligned}
					&	\int_{X}c(x_1,F_2(x_1),\dots, F_n(x_1))\,\d\mu_1(x_1)\\=&\int_{X}\sum_{i=1}^n \ds^2(F_i(x_1),S(x_1))\d {\mu_1}(x_1)\\
					\geq& \sum_{i=1}^n W_2^2(\mu_i,S_\sharp \mu_1) \\
						\overset{\eqref{5.5}}\geq &\inf_{\pi\in\Pi}\int_{X^n}c(x_1,\dots,x_n)\,\d\pi.
					\end{aligned}
	\end{equation}
By optimality of $(F_2,\dots,F_n)$,  we know inequalities above are equalities.   By uniqueness of the Wasserstein barycenter we know $S_\sharp \mu_1=\bar \mu$.   By  uniqueness of the optimal transport maps from $\bar \mu$ to $\mu_i$,  we know $S=T_1^{-1}$ and $F_i=T_i\circ T_1^{-1}$ which is the contradiction.
	
	
\end{proof}

Next, we will prove our main theorem in this subsection,   about the unique resolvability  of the  multi-marginal optimal transport problem of Monge type, for absolute continuous marginals,  \emph{in full generality}.

\begin{theorem}[Existence and uniqueness of multi-marginal optimal transport map]\label{MOT}
	Let $(X,\ds, \mm)$ be a metric measure space. Then the multi-marginal optimal transport problem of Monge type, associated with the cost function $c(x_1,\dots,x_n)$,   has a unique solution,  if one of the following conditions holds:
	\begin{itemize}
	\item [{\bf A.}] $(X,\ds,\mm)$ is an ${\rm RCD}(K,N)$ space, and $\mu_1\ll \mm$;
	
	\item [{\bf B.}] $(X,\ds,\mm)$ is an ${\rm RCD}(K,\infty)$ space, and $\mu_i\ll \mm,i=1,\dots,n$.
	\end{itemize}
	In particular, any multi-marginal  optimal transport plan $\pi$ is concentrated on the graph of a $X^{n-1}$-valued map, and for $\pi$-a.e. ${\bf x}:=(x_1, \dots, x_n)\in X^n$,  the measure $\frac 1n \sum_{i=1}^n \delta_{x_i}$  has a unique barycenter in $X$.
\end{theorem}

\begin{proof}
	We will prove the existence by contradiction, the uniqueness can be proved in the same way as Proposition \ref{finite entropy MOT}.
	
	Assume  there exists a multi-marginal optimal plan $\pi$, which is not induced by a multi-marginal  transport map. Then there is  $E\subset X^n$ with positive $\pi$-measure, such that any subset $E'\subset E$ with positive measure is not included in the graph of a map from $X$ to $X^{n-1}$.
	
{For  $\rcdkn$ spaces,}  assume $\mu_1=\rho_1 \mm\ll \mm$, the union $\bigcup_{C> 0}\{x_1\in X: {C}^{-1}\leq \rho_1(x_1)\leq C\}$ has full $\mu_1$-measure. Therefore, there exists some $C>0$, such that the set $\tilde E:=\{(x_1,\dots, x_n)\in E: {C}^{-1}\leq\rho_1(x_1)\leq C\}$ has positive measure. Consider the plan $\tilde \pi:= \frac 1 {\pi(\tilde E)}\pi \llcorner_{\tilde E}$. By optimality of $\pi$ and the linearity  of the Kantorovich problem,  we can see that $\tilde \pi$ is also a multi-marginal optimal transport plan.
	By Proposition \ref{finite entropy MOT} (\romannumeral1),  $\tilde \pi$ is concentrated on the graph of a $X^{n-1}$-valued map,  which contradicts to the choice of $E$.
	
	{Concerning $\rcd$ spaces,   by induction and a similar truncation argument as above, we can reduce the problem to a multi-marginal optimal transport problem with  marginal measures having finite entropy, then by Proposition \ref{finite entropy MOT} (\romannumeral2)  we get a contradiction.}
\end{proof}

Using  Theorem \ref{MOT}, we can remove the   finite entropy condition from the hypothesis of Proposition \ref{finite support}.

\begin{theorem}\label{absolutely continuous}
	Let $(X,\ds,\mm)$ be a metric measure space. Assume $\mu_1,\dots,\mu_n\in \pr_2(X, \ds)$, then       $\frac 1n \sum_{i=1}^n \delta_{\mu_i}$  has a unique Wasserstein barycenter $\bar{\mu}$ and  $\bar{\mu} \ll \mm$,  if one of the following conditions holds:
	\begin{itemize}
	\item [{\bf A.}] $(X,\ds,\mm)$ is an ${\rm RCD}(K,N)$ space, and $\mu_1\ll \mm$;
	\item [{\bf B.}]  $(X,\ds,\mm)$ is an ${\rm RCD}(K,\infty)$ space, and $\mu_i\ll \mm,i=1,\dots,n$.
	\end{itemize}
\end{theorem}

\begin{proof}
 Consider the multi-marginal optimal transport problem associated with  measures $\mu_1, \dots, \mu_n$. By Theorem \ref{MOT}, there exists a unique multi-marginal optimal transport map $(T_2,\dots,T_n)$,  such that  for $\mu_1$-a.e. $x\in X$,  the measure $\frac 1n (\sum_{i=2}^n \delta_{T_i(x)}+\delta_x)$  has a unique barycenter in $X$. Then by Theorem \ref{barycenter space} we know 
 $\frac 1n \sum_{i=1}^n \delta_{\mu_i}$ has a barycenter in the Wasserstein space.
 
  Now, denote by $\bar{\mu}$ the Wasserstein barycenter.  It is sufficient to prove  $\bar{\mu}\ll\mm$, then by strict convexity of the Wasserstein distance with respect to the linear interpolation (cf. Proposition \ref{uniqueness in finite space} and Theorem \ref{uniqueness in infinite space}), we can  prove the uniqueness of the Wasserstein barycenter.

{For $\rcdkn$ spaces, } denote $\mu_1=\rho_1 \mm$.
	Let $ F\subseteq X$ be such that $\mm(F)=0$. Assume by contradiction that $\bar{\mu}(F)>0$.   Denoted by $S$ the unique optimal transport map from $\mu_1$ to $\bar{\mu}$.  For $E=S^{-1}(F)$,   we have $\bar \mu(F)=\mu_1(E)>0$.
So there exists $ E^1 \subset E$ so that $\mu_1(E^1)>0$ and $\|\rho_1\|_{L^\infty(E^1, \mm)}<\infty$.
Set $\nu_1=\frac 1{\mu_1(E^1)} \mu_1\llcorner_ {E^1}, \nu_k={T_k}_{\sharp} \nu_1$, $k=2,\dots,n$. By construction, $\nu_1,\dots,\nu_n\in \pr_2(X,\ds)$ and ${\rm Ent}_\mm(\nu_1)<\infty$.  By Proposition \ref{finite support},     $\frac 1n \sum_{i=1}^n \delta_{\nu_i}$ has a unique  Wasserstein barycenter $\bar{\nu}\in \pr(X)$ and $\bar{\nu}\ll \mm$,  so  $\bar{\nu}(F)=0$.  However,  by Theorem \ref{barycenter space}  and the uniqueness of the Wasserstein barycenter,  we  also have $\bar{\nu}(F)=1$ which is a contradiction. Thus $\bar{\mu}(F)=0$, and by arbitrariness of $F$,  we have ${\mu}\ll \mm$.
	
	{Concerning $\rcd$ spaces,   by induction and a similar truncation argument as above, we can reduce the problem to a Wasserstein barycenter problem  associate with   measures having finite entropy,  then by Proposition \ref{finite support} we can prove the assertion.} 
\end{proof}

\section{${\rm BCD}$ condition}\label{bcd}
\subsection{Definition}
 There is no doubt that  curvature is one of the most important  concepts in  geometry. In particular, in the study of    metric  geometry, it has been a long history to give a synthetic notion of upper and lower curvature bounds.  Metric spaces with sectional curvature  lower bound, known as Alexandrov spaces, has been widely studied and gain great achievement  in the last century.  About twenty years ago, Sturm and Lott--Villani independently introduced a  synthetic notion of lower Ricci curvature bounds, called Lott--Sturm--Villani curvature--dimension condition today, in the general framework of   metric measure spaces.  This Lott--Sturm--Villani curvature-dimension condition is defined  in terms of McCann's displacement convexity,  of certain functionals on the Wasserstein space. In his celebrated paper  \cite{mccann1997convexity}, McCann  introduced  the notion of displacement convexity,   on  the Wasserstein space over $\mathbb{R}^n$. Later,  Cordero-Erausquin, McCann and Schmuckenschl{\"a}ger \cite{CMS01} extended this result to Riemannian manifolds with Ricci curvature lower bound.  Conversely,  von Renesse and Sturm \cite{SVR-T} proved that the displacement convexity actually implies lower Ricci curvature bounds. Thus, in the setting of Riemannian manifolds, lower Ricci curvature bounds can be characterized by the displacement convexity of certain functionals. This fact is actually one of the main motivations  of Lott--Sturm--Villani's theory.

In this section,  based on the barycenter convexity  of certain functionals on the Wasserstein spaces,  proved in the last Section \ref{main},  we introduce a new curvature-dimension condition, called {{B}arycenter-{C}urvature-{D}imension condition}.

Motivated by Theorem \ref{WJI}, we introduce a dimension-free, synthetic,   curvature-dimension condition for \emph{extended metric measure  spaces}:

\begin{definition}[${\rm BCD}(K,\infty)$  condition]\label{def:bcd}
	
	Let $K \in \R$.  We say that an {extended metric measure  space} $(X,\ds,\mm)$ verifies ${\rm BCD}(K,\infty)$ condition,   if for any  probability measure $\Omega\in \pr_2(\pr(X), W_2)$ with finitely support, there {exists} a barycenter $\bar{\mu}$ of  $\Omega$ such that the following Jensen's  inequality holds:
	\begin{equation}\label{eq:defBCD infty}
	{\rm Ent}_{\mm}(\bar{\mu})\leq\int_{\mathcal{P}(X)}{\rm Ent}_{\mm}(\mu)\,\d\Omega(\mu)-\frac{K}{2}{\rm Var}(\Omega).
	\end{equation}

\end{definition}
\begin{remark}
If $(X, \ds)$ is  geodesic,  taking $\Omega=(1-t)\delta_{\mu_0}+t\delta_{\mu_1}$ in \eqref{eq:defBCD infty} we get the displacement convexity.  So in this case,  BCD condition implies Lott--Sturm--Villani's curvature-dimension condition. 
However,  as we have seen,   many examples of {\rm BCD} spaces are not geodesic.  It is also interesting to find more {\rm BCD} space  which is not geodesic.
\end{remark}

\bigskip

Motivated by Theorem \ref{JIEVI -K,N}, we can also define a finite dimensional variant of barycenter curvature-dimension condition:

\begin{definition}[${\rm BCD}(K,N)$  condition]\label{def:bcdkn}
	
	Let $K \in \R, N>0$.  We say that a   {metric measure  space}  $(X,\ds,\mm)$ verifies ${\rm BCD}(K,N)$ condition,   if for any  probability measure $\Omega\in \pr_2(\pr(X), W_2)$ with finitely support, there exists a barycenter $\bar{\mu}$ of  $\Omega$ such that the following Jensen-type  inequality holds:
\begin{equation}\label{eq:defbCDKN}
\int \frac{W_2(\bar{\mu},\mu)}{s_{K/N}(W_2(\bar{\mu},\mu))}U_N(\mu)\,\d\Omega(\mu)\leq  U_N(\bar{\mu}) \int \frac{W_2(\bar{\mu},\mu)}{t_{K/N}(W_2(\bar{\mu},\mu))}\,\d\Omega(\mu),
\end{equation}
where $U_N(\mu)=e^{-\frac{{\rm Ent}_{\mm}(\mu)}{N}}$.
\end{definition}

\subsection{Stability in measured Gromov--Hausdorff topology}
Similar to  \cite[Theorem 4.15]{Lott-Villani09}, we can prove the stability of ${\rm BCD}$ condition  in measured Gromov--Hausdorff convergence.    For simplicity,   in this paper we focus on compact ${\rm BCD}(K, \infty)$ metric measure spaces, general cases will be studied in a forthcoming paper.

First of all,   we  recall the definition of  measured  Gromov\textendash  Hausdorff convergence.    It is known  that this convergence   is compatible with a metrizable topology on the space of all compact metric spaces modulo isometries.     We refer the readers to \cite{BBI-A}, \cite{GMS-C}
 and  \cite{V-O}  for equivalent definitions and Gromov's book \cite{Gromovbook}  for more  about  the convergence of non-compact metric measure spaces.

\begin{definition} Given $\epsilon \in (0,1)$. A map $\varphi : (X_1, \ds_1) \to  (X_2,  \ds_2)$ between two metric  spaces is called an $\epsilon$-approximation if
\begin{itemize}
\item[(i)] $\Big |\ds_2\big (\varphi (x), \varphi (y)\big) - \ds_1(x, y)\Big| \leq  \epsilon$, ~~~ $\forall   x, y \in  X_1$;
\item[(ii)]
For all $x_2 \in X_2$, there is an $x_1\in X_1$ so that $\ds_2\big(\varphi(x_1), x_2\big)\leq \epsilon$.
\end{itemize}

\end{definition}

\begin{definition}[Measured Gromov\textendash Hausdorff convergence]\label{def:conv}
We say that a sequence of  metric spaces  $\{(X_i, \ds_i)\}_{i\in \mathbb{N}}$   converges in  Gromov\textendash  Hausdorff sense to $(X, \ds)$ if there exists a sequence  of $\epsilon_i$-approximations   $\varphi_i: X_n \to X$ with $\epsilon_i \downarrow 0$.

We say that a sequence of  metric measure spaces $(X_i, \ds_i, \mm_i)$  converges in  measured  Gromov\textendash  Hausdorff sense (mGH for short) to a  metric measure space $(X,  \ds, \mm)$, if additionally $(\varphi_i)_\sharp \mm_i \to  \mm$ weakly as measures, i.e. 
\[
\lim_{i \to \infty} \int_X g \, \d \big  ((\varphi_i)_\sharp \mm_i\big) = \int_X g \, \d \mm \qquad \forall g \in C_b(X), 
\]
where $C_b(X)$ denotes the set of  bounded continuous functions.
\end{definition}

\medskip

\begin{theorem}\label{measured gromov hausdorff limit}
	Let $K\in\mathbb{R}$ and $\left\{(X_i,  \ds_i, \mm_i)\right\}_{i\in \mathbb{N}}$ be a sequence of compact ${\rm BCD}(K,\infty)$ metric measure spaces converging to $(X, \ds, \mm)$ in the measured Gromov--Hausdorff sense.  Then $(X, \ds, \mm)$ is also a ${\rm BCD}(K,\infty)$ space. 
\end{theorem}
\begin{proof}
	It is  known that compactness is stable under the Gromov--Hausdorff convergence. Thus, $(X,\ds)$ is  barycenter space.  Let $\Omega\in\mathcal{P}_2(\mathcal{P}(X),W_2)$ be with finite support and let $\bar{\mu}$ be a barycenter of $\Omega$.   By \cite[Lemma 3.24]{Lott-Villani09}, we may assume that  $\Omega=\sum_{k=1}^{m}\lambda_k\delta_{\mu_k}$ with $\sum_{k=1}^{m}\lambda_k=1$ and
	$$
	\mu_k\in \Big\{\gamma\in \mathcal{P}(X):{\rm Ent}_{\mm}(\gamma)<\infty, \gamma \text{ has bounded continuous density} \Big\}.
	$$

	Let  $f_i:X_i\rightarrow X$ and $f_i^{'} :X\rightarrow X_i$ be  $\epsilon_i$-approximations with $\lim_{i\rightarrow\infty}\epsilon_i=0$   and $\lim_{i\rightarrow\infty}(f_i)_{\sharp}\mm_i=\mm$.  Assume  $\mu_k=\rho_k\mm$ with $\rho_k\in C_b(X)$.  For $i$  sufficiently large, we have $\int_{X}\rho_k \,\d(f_i)_{\#}\mm_i>0$. For such $i$, put $\mu_{i,k}=\frac{(f_i^{\sharp}\rho_k)\mm_i}{\int_{X}\rho_k \,\d(f_i)_{\sharp}\mm_i}$ and $\Omega_i=\sum_{k=1}^{m}\lambda_k\delta_{\mu_{i,k}}$,  where $f_i^{\sharp}\rho_k:=\rho_k \circ f_i$ denotes the pull-back of the function $\rho_k$. Let $\bar{\omega}_i$ be a barycenter of $\Omega_i$.  We have the following properties, whose proof  can be found in \cite[Corollary 4.3 and Theorem 4.15]{Lott-Villani09}:
	\begin{enumerate}
	\item [(i)]  for any $\gamma\in \mathcal{P}(X)$, $\lim_{i\rightarrow\infty}W_2(\mu_{i,k},(f_i^{'})_\sharp\gamma)=W_2(\mu_k,\gamma)$;
	\item [(ii)] $\lim_{i\rightarrow\infty}{\rm Ent}_{\mm_{i}}(\mu_{i,k})={\rm Ent}_{\mm}(\mu_k)$;
	\item  [(iii)] ${\rm Ent}_{(f_i)_{\sharp}\mm_i}((f_i)_{\sharp}\bar{\omega}_i)\leq {\rm Ent}_{\mm_i}(\bar{\omega}_i)$;
	\item [(iv)] the functional $(\nu, \mu) \mapsto {\rm Ent}_{\nu}(\mu)$ is lower semi-continuous.
	\end{enumerate}

	We claim that $(f_i)_\sharp \bar{\omega}_i$ converges to a barycenter $\bar{\omega}$ of $\Omega$ up to taking a subsequence. 
	
	By property (i) above,  for any $\gamma\in\mathcal{P}(X)$,
	$$
	\sum_{k=1}^m \lambda_k W_2^2(\mu_k,\gamma)=\lim_{i\rightarrow\infty}\sum_{k=1}^m \lambda_k W_2^2(\mu_{k,i},(f_i^{'})_\sharp\gamma)\geq\limsup_{i\rightarrow\infty} {\rm Var}(\Omega_i).
	$$ 
This implies that ${\rm Var}(\Omega)\geq\limsup_{i\rightarrow\infty} {\rm Var}(\Omega_i)$. 
By taking a subsequence,  we assume that $(f_i)_{\sharp}(\bar{\omega_i})$ converges to  $\bar{\omega} \in \pr(X)$.
Notice that
	$$
	\lim_{i\rightarrow\infty} W_2(\mu_k,(f_i)_\sharp \bar{\omega}_i)= 	\lim_{i\rightarrow\infty} W_2(\mu_{k,i}, \bar{\omega}_i)= W_2(\mu_{k}, \bar{\omega})~~~\forall k=1,2,...,m.
	$$We have 
\[
\begin{aligned}
{\rm Var}(\Omega)&\geq\limsup_{i\rightarrow\infty} {\rm Var}(\Omega_i)=\limsup_{i\rightarrow\infty}\sum_{k=1}^m \lambda_k W_2^2(\mu_k,(f_i)_\sharp \bar{\omega}_i)\\&=\sum_{k=1}^m \lambda_k W_2^2(\mu_k, \bar{\omega})\geq {\rm Var}(\Omega).
\end{aligned}
\]
So $\bar{\omega}$  is a barycenter of $\Omega$.

 In conclusion, we get
	\begin{equation}
	\begin{aligned}
	{\rm Ent}_{\mm}(\bar{\omega})&\overset{\text{(iv)}}{\leq} \liminf_{i\rightarrow\infty}{\rm Ent}_{(f_i)_{\sharp}\mm_i}(f_i)_{\sharp}(\bar{\omega_i})
	\overset{\text{(iii)}}{\leq}\liminf_{i\rightarrow\infty}{\rm Ent}_{\mm_i}(\bar{\omega}_i)\\
	&\overset{\eqref{eq:defBCD infty}}{\leq} \liminf_{i\rightarrow\infty}\int_{\mathcal{P}(X_i)}{\rm Ent}_{\mm_i}(\mu)\,\d\Omega_i(\mu)-\frac{K}{2}{\rm Var}(\Omega_i)\\
	&\overset{\text{(ii)}}{=}\int_{\mathcal{P}(X)}{\rm Ent}_{\mm}(\mu)\,\d\Omega(\mu)-\frac{K}{2}{\rm Var}(\Omega)
	\end{aligned}
	\end{equation}
	which is the thesis.
\end{proof}

\subsection{Wasserstein barycenter in ${\rm BCD}(K,\infty)$ spaces}
In this section, we will prove the existence of Wasserstein barycenter in ${\rm BCD}(K,\infty)$ spaces under mild assumptions.

\begin{theorem}\label{Existence of Wasserstein barycenter in barycenter space}
	Let $(X,\d,\mm)$ be an extended metric measure space satisfying ${\rm BCD}(K,\infty)$ curvature-dimension condition, $\Omega$ be a probability measure on $\pr(X)$ satisfying $${\rm Var}(\Omega)<\infty~~~~\text{and}~~~\int_{\pr(X)}{\rm Ent}_\mm(\mu)\,\d\Omega(\mu)<\infty.$$ Then  $\Omega$ has a barycenter if one of the following conditions holds:
	\begin{itemize}
		\item [{\bf A.}]  $(X, \ds, \mm)$ satisfies the exponential volume growth condition and $\Omega$ is concentrated on $\mathcal{P}_2(X, \ds)$ (cf. Theorem \ref{rcdevi});
		\item [{\bf B.}]  $\mm$ is a probability measure.
	\end{itemize} 
\end{theorem}

\begin{proof}
By \cite[Theorem 6.18]{V-O}, for any $\Omega\in \pr_2(\pr(X), W_2)$,  we can find a sequence of  probability measures $\Omega_j\in \pr_0(\pr(X))$, such that $\Omega_j \to \Omega$  in  Wasserstein distance as $j\to\infty$. By Definition \ref{def:bcd},   $\Omega_j$ admits a Wasserstein barycenter  $\mu_j$. By stability of the  Wasserstein barycenters (cf.  \cite[Theorem 3]{le2017existence}), it is sufficient to show  that $(\mu_j)_{j\in \N}$ has a narrow convergent subsequence. We divide  the proof into three steps.

	\paragraph{Step 1:} We claim there exists a sequence $(\Omega_j )\subset \pr_0(\pr(X))$, such that $\Omega_j \to \Omega$ and $\lmts{j}{\infty}\int _{\pr(X)} {\rm Ent}_\mm(\mu)\, \d\Omega_{j}(\mu)\leq \int _{\pr(X)} {\rm Ent}_\mm(\mu) \,\d\Omega(\mu)$. 
	
	To prove this, we adopt an argument used in\cite[Theorem 6.18]{V-O}.  For any $j\geq 1$,  set $\epsilon=1/j$, and let $\mu_0\in \pr(X)$ be such that ${\rm Ent}_\mm(\mu_0)\leq 1$.  Note that $\Omega$ is Radon,  there exists a compact subset $\mathsf K\subseteq \pr(X)$, such that $\Omega(\pr(X)\backslash \mathsf K)<\epsilon$ and
	$$\int_{\pr(X)\backslash \mathsf K} W_2^2(\mu_0,\mu)\,\d \Omega(\mu)<\epsilon.$$
	Cover $\mathsf K$ by finite many  geodesic balls $B_{\epsilon}(\mu_k),1\leq k\leq N$. Define $$B_k^{\prime}=B_\epsilon(\mu_k)\backslash \bigcup_{j<k}B_\epsilon(\mu_j).$$
	By construction,   $\{B_k^{\prime}\}_{1\leq k\leq N}$ is a disjoint  covering of $\mathsf K$. Without loss of generality, we assume that $\Omega(B_k^{\prime}\cap \mathsf K)>0$ for any $k$.  Note that for any $B_k^{\prime}\cap \mathsf K$, there exists $\mu_k^{\prime}\in B_k^{\prime}\cap \mathsf K$, such that 
	\begin{equation}\label{6.4}
		{\rm Ent}_\mm(\mu_k^{\prime})\leq \frac{1}{\Omega(B_k^{\prime}\cap \mathsf K)} \int_{B_k^{\prime}\cap \mathsf K} {\rm Ent}_\mm(\mu)\,\d \Omega(\mu).
	\end{equation}
Define  $f_j:\pr(X) \to \pr(X)$ by
	\begin{equation}
		f_j(B_k^{\prime}\cap \mathsf K)=\{\mu_k^{\prime}\},  \quad f(\pr(X)\backslash \mathsf K)=\{\mu_0\}.
	\end{equation}
	Then, for any $\mu\in \mathsf K$,  there is $\mu'_k\in \mathsf K$ such that $W_2(\mu, f_j(\mu))=W_2(\mu, \mu'_k)\leq \epsilon$. So
\begin{equation}\label{5.24}
		\begin{aligned}
			\int_{\pr(X)} W_2^2(\mu,f_j(\mu))\,\d \Omega(\mu)
			&=\int_{\mathsf K} W_2^2(\mu,f_j(\mu))\,\d \Omega(\mu)+\int_{\pr(X)\backslash \mathsf K}W_2^2(\mu,f_j(\mu))\,\d \Omega(\mu)\\
			&\leq \epsilon^2\int_{\mathsf K}\, \d\Omega(\mu)+\int_{\pr(X)\setminus \mathsf K}W_2^2(\mu,\mu_0)\,\d \Omega(\mu)\\
			&\leq \epsilon^2+\epsilon.
		\end{aligned}
	\end{equation}
	This implies $W_2(\Omega, (f_j)_{\sharp}\Omega)\leq \epsilon^2+\epsilon$, where $(f_j)_{\sharp}\Omega\in \pr_0(\pr(X))$. Moreover, note that 
	\begin{equation}
		\begin{aligned}
			\int _{\pr(X)} {\rm Ent}_\mm(\mu)\, \d (f_j)_{\sharp}\Omega(\mu)
			&=\int _{\mathsf K}  {\rm Ent}_\mm(f_j(\mu)) \,\d \Omega(\mu)+\int _{\pr(X)\backslash \mathsf K}  {\rm Ent}_\mm(f_j(\mu)) \,\d \Omega(\mu)\\
			&= \sum_{k=1}^N \int _{B_k^{\prime}\cap \mathsf K}  {\rm Ent}_\mm(f_j(\mu)) \,\d \Omega(\mu)+\int _{\pr(X)\backslash \mathsf K}  {\rm Ent}_\mm(f_j(\mu))\, \d \Omega(\mu)\\
			&=\sum_{k=1}^N \int _{B_k^{\prime}\cap \mathsf K}  {\rm Ent}_\mm(\mu_k^{\prime}) \,\d \Omega(\mu)+\int _{\pr(X)\backslash \mathsf K}  {\rm Ent}_\mm(\mu_0) \,\d \Omega(\mu)\\
			&\overset{\eqref{6.4}}{\leq }\sum_{k=1}^N  \int_{B_k^{\prime}\cap \mathsf K}  {\rm Ent}_\mm(\mu)\,\d \Omega(\mu)+\int _{\pr(X)\backslash \mathsf K}  {\rm Ent}_\mm(\mu_0)\, \d \Omega(\mu)\\
			&\leq \int _{\mathsf K}  {\rm Ent}_\mm(\mu) \,\d\Omega(\mu)+\epsilon.
		\end{aligned}
	\end{equation}
Define $\Omega_j=(f_j)_{\sharp}\Omega$.  As $j\to \infty$, we  have $\Omega_j\to\Omega$.  Letting $j\to \infty$  and  $\mathsf K\to \pr(X)$,  we get $$\lmts{j}{\infty}\int _{\pr(X)}  {\rm Ent}_\mm(\mu) \,\d\Omega_{j}(\mu)\leq \int _{\pr(X)}  {\rm Ent}_\mm(\mu)\, \d\Omega(\mu).$$
	
	\paragraph{Step 2:}  Since  $\Omega_j$ is concentrated on finite number of probability measures, by Definition  \ref{def:bcd},  $\Omega_j$ admits a Wasserstein barycenter $\mu_j$, such that
	$$\int_{\pr(X)} W_2^2(\mu,\mu_j)\,\d\Omega_j(\mu)=\min_{\nu\in \pr(X)}\int_{\pr(X)} W_2^2(\mu, \nu)\,\d \Omega_j(\mu).$$
	Notice that  ${\rm Var}(\Omega) <+\infty$, so  there is $\nu\in \pr(X)$ such that
\[
\int_{\pr(X)} W_2^2(\nu,\mu)\,\d\Omega(\mu)<+\infty.
\]
	 {\bf Claim}: it holds the following estimate
	\begin{equation}
	\lmts{j}{\infty}	\int_{\pr(X)} W_2^2(\mu,\mu_j)\,\d\Omega_j(\mu)\leq 2 \int_{\pr(X)} W_2^2(\nu,\mu)\,\d \Omega(\mu)<+\infty.
	\end{equation}
	
	Note that by triangle inequality, for any $\mu\in \pr(X)$,
	\begin{equation}\label{5.14}
		W_2^2(\nu,f_j(\mu))\leq 2\big(W_2^2(\nu,\mu)+W_2^2(\mu,f_j(\mu))\big).
	\end{equation}
	Integrating \eqref{5.14} with respect to $\Omega$, we obtain
	\begin{equation}\label{5.15}
		\begin{aligned}
			&\int_{\pr(X)} W_2^2({\mu_j},\mu)\,\d\Omega_j(\mu)\\
			&\leq \int_{\pr(X)} W_2^2({\nu},\mu)\,\d\Omega_j(\mu)=\int_{\pr(X)} W_2^2(\nu,f_j(\mu))\,\d\Omega(\mu)\\
			&\leq 2\left(\int_{\pr(X)} W_2^2(\nu,\mu)\,\d\Omega(\mu)+\int_{\pr(X)} W_2^2(\mu,f_j(\mu))\,\d\Omega(\mu)\right)\\
			&\overset{\eqref{5.24}}\leq  2\left(\int_{\pr(X)} W_2^2(\nu,\mu)\,\d\Omega(\mu)+ \epsilon^2+\epsilon\right).
		\end{aligned}
	\end{equation}

Letting $j\to \infty$ and $\epsilon\rightarrow0^{+}$ in \eqref{5.15}, we prove this claim.

	\paragraph{Step 3:} By  assumption, there is $c\geq 0$ such that
$$
\int_X e^{-c\ds(x_0,x)^2}\,\d\mm(x)<\infty,\qquad \text{for all } x_0 \in X,
$$
	define $z=\int_{X} e^{-c\ds^2(x_0,x)}\,\d\mm(x)$ and $\tilde{\mm}=\frac{1}{z}e^{-c\ds^2(x_0,x)}\mm\in \pr(X)$. By changing the  reference measure, we get
	\begin{equation}
		{\rm Ent}_\mm(\mu)={\rm Ent}_{\tilde{\mm}}(\mu)-c\int_{X} \ds^2(x_0,x)\,\d\mu-\ln z.
	\end{equation}
	Note that by Jensen's inequality  \eqref{eq:defBCD infty}, for every $j\in\mathbb{N}$, 
	$${\rm Ent}_\mm(\mu_j)\leq \int _{\pr(X)} {\rm Ent}_\mm(\mu)\, \d\Omega_j(\mu)-\frac{K}{2}\int_{\pr(X)}W_2^2(\mu_j,\mu)\,\d\Omega_j(\mu).$$
	Therefore,  for $K\geq 0$ we have 
	\begin{eqnarray*}
			{\rm Ent}_\mm(\mu_j)\leq \int _{\pr(X)} {\rm Ent}_\mm(\mu) \,\d\Omega_j(\mu),
			\end{eqnarray*}
			and for $K<0$ we have
			\begin{eqnarray*}
			{\rm Ent}_\mm(\mu_j)\leq \int _{\pr(X)} {\rm Ent}_\mm(\mu) \,\d\Omega_j(\mu)-K\int_{\pr(X)} W_2^2(\mu_0,\mu)\,\d \Omega(\mu).
			\end{eqnarray*}
	Combining with Step 1 and Step 2,  we known there is $C=C(c,K)>0$, so that
\begin{equation}
		{\rm Ent}_{\tilde{\mm}}(\mu_j)\leq   C<\infty,
	\end{equation}
for every $j\in\mathbb{N}$. This surely implies  the sequence $(\mu_j)_{j\geq 1}$
	is tight. By Prokhorov's theorem, there exists a narrow convergent subsequence of $(\mu_j)_{j\geq 1}$. Then by \cite[Theorem 3]{le2017existence}, the limit of the narrow convergent subsequence is a Wasserstein barycenter of $\Omega$. We complete the proof.
\end{proof}

\subsection{Applications}\label{sect:app}
At the end of this paper, we prove  some new geometric inequalities,  as simple but interesting applications of our BCD theory. For simplicity, we  will only deal with BCD metric measure spaces. More functional and geometric inequalities on BCD extended metric measure spaces, will be studied in a forthcoming paper. We also refer the readers to a paper \cite{KW-BS} by Kolesnikov and Werner  for more applications.

\begin{proposition}[Multi-marginal  Brunn--Minkowski inequality]
Let $\ms$ be a ${\rm BCD}(0,N)$ metric measure space, $E_1,..., E_n$ be bounded measurable sets with positive measure and $\lambda_1,\dots, \lambda_n\in (0, 1)$ with $\sum_i \lambda_i=1$. Then
\[
\mm(E)\geq \left(\sum_{i=1}^n \lambda_i \big(\mm(E_i)\big)^{\frac 1N}\right)^N
\]
where 
$$
E:=\left\{x ~\text{is the barycenter of}~\sum_{i=1}^n \lambda_i\delta_{x_i}: x_i\in E_i, i=1,\dots, n \right\}.
$$ 
\end{proposition}
\begin{proof}
Set $\mu_i:=\frac 1{\mm(E_i)}\mm\llcorner_{E_i}, i=1,\dots, n$ and $\Omega:=\sum_{i=1}^n \lambda_i\delta_{\mu_i}\in \pr_0(\pr_2(X, \ds))$.
By Definition \ref{def:bcdkn}, \eqref{eq:defbCDKN} we get
\begin{equation}
\int U_N(\mu)\,\d\Omega(\mu)\leq  U_N(\bar{\mu}),
\end{equation}
where $\bar{\mu}$ is the	barycenter  of  $\Omega$.
This implies
\[
\sum_{i=1}^n \lambda_i(\mm(E_i))^{\frac 1N} \leq U_N(\bar{\mu}).
\]

 Note that by Theorem \ref{th:sp}, $\bar{\mu}$ is concentrated on $E$. So by Jensen's inequality, we have
\[
{\rm Ent}_{\mm}(\bar{\mu}) \geq -\ln(\mm(E))
\]
and
\[
U_N(\bar{\mu}) \leq  (\mm(E))^{\frac 1N}.
\]
In conclusion, we  obtain
\[
\sum_{i=1}^n \lambda_i(\mm(E_i))^{\frac 1N} \leq (\mm(E))^{\frac 1N}.
\]
which is the thesis.
\end{proof}

\medskip

\begin{proposition}[Multi-marginal logarithmic Brunn--Minkowski inequality]
Let $\ms$ be a ${\rm BCD}(0,\infty)$ metric measure space,  $E_1,..., E_n$ be  bounded measurable sets with positive measure  and $\lambda_1,\dots, \lambda_n\in (0, 1)$ with $\sum_i \lambda_i=1$. Then
\[
\mm(E) \geq \mm(E_1)^{\lambda_1}... \mm(E_n)^{\lambda_n}
\]
where 
$$
E:=\left\{x ~\text{is the barycenter of}~\sum_{i=1}^n \lambda_i\delta_{x_i}: x_i\in E_i, i=1,\dots, n \right\}.
$$
\end{proposition}
\begin{proof}
Set $\mu_i:=\frac 1{\mm(E_i)}\mm\llcorner_{E_i}, i=1,\dots, n$ and $\Omega:=\sum_{i=1}^n \lambda_i\delta_{\mu_i}\in \pr_0(\pr_2(X, \ds))$.
By Definition \ref{def:bcd}, \eqref{eq:defBCD infty} we get
\[
	{\rm Ent}_{\mm}(\bar{\mu}) \leq -\sum_{i=1}^n \lambda_i \ln(\mm(E_i)).
	\]
where $\bar{\mu}$ is the	barycenter  of  $\Omega$. By Theorem \ref{th:sp} we can see that $\bar{\mu}$ is concentrated on $E$.   By Jensen's inequality, we have
\[
{\rm Ent}_{\mm}(\bar{\mu}) \geq -\ln(\mm(E)).
\]
In conclusion, we  obtain
\[
-\ln(\mm(E)) \leq -\sum_{i=1}^n \lambda_i \ln(\mm(E_i))
\]
which is the thesis.
\end{proof}

\medskip

\begin{proposition}[A     functional Blaschke--Santal\'o type  inequality]
Let $\ms$ be  a ${\rm BCD}(1,\infty)$ metric measure space. Then we  have
\begin{equation*}\label{eq1:th3}
\prod_{i=1}^k \int_X e^{f_i}\,\d \mm\leq 1 
\end{equation*}
for any  measurable functions   $f_i$ on $X$ such that $\frac {e^{f_i}}{\int e^{f_i}\,\d \mm}\in \pr_2(X, \ds)$ and
\begin{equation*}\label{eq:duality}
\sum_{i=1}^k f_i(x_i)\leq  \frac 1{2} \inf_{x\in X} \sum_{i=1}^k  \ds^2(x, x_i)\qquad \forall x_i\in X, i=1,2,...,k.
\end{equation*}
\end{proposition}
\begin{proof}
 Let  $\mu_i:=\frac {e^{f_i}}{\int e^{f_i}\,\d \mm},i=1,2,...,k$ be probability measures on $X$.  Let $\mu$ be the Wasserstein barycenter of the probability measure $\sum_{i=1}^n  
 \frac 1k \delta_{\mu_i}$.  By Definition \ref{def:bcd},  it holds the following  Wasserstein Jensen's inequality
 \[
 \ent \mm(\mu)\leq  \sum_{i=1}^k \frac 1k \ent \mm(\mu_i) -\frac 1{2k}  \sum_{i=1}^k  W_2^2(\mu, \mu_i).
 \]
Therefore
\begin{eqnarray*}
&&\sum_{i=1}^k \int f_i(x_i)\,\d \mu_i(x_i)\leq \frac 12 \inf_{\pi\in \Pi(\mu_1,\dots,\mu_k)} \int\inf_{x\in X} \sum_{i=1}^k  \ds^2(x, x_i)\,\d \pi\\
&=&\frac 12  \sum_{i=1}^k  W_2^2(\mu, \mu_i) \leq  \sum_{i=1}^k  \ent \mm(\mu_i)- k\ent \mm(\mu)\\&\leq &  \sum_{i=1}^k \Big(\int \frac {e^{f_i}}{\int e^{f_i}} \ln \frac {e^{f_i}}{\int e^{f_i}}\,\d \mm\Big)\\
&=&  \sum_{i=1}^k \Big( \int f_i(x_i)\,\d \mu_i(x_i)-\ln \int e^{f_i}\,\d \mm\Big).
\end{eqnarray*}
So  $ \sum_{i=1}^k \ln \int e^{f_i}\,\d \mm\leq 0$ which is the thesis.
\end{proof}

\addcontentsline{toc}{section}{References}
\def\cprime{$'$}
\providecommand{\bysame}{\leavevmode\hbox to3em{\hrulefill}\thinspace}
\providecommand{\MR}{\relax\ifhmode\unskip\space\fi MR }
\providecommand{\MRhref}[2]{%
  \href{http://www.ams.org/mathscinet-getitem?mr=#1}{#2}
}
\providecommand{\href}[2]{#2}

\end{document}